\documentclass[11pt,leqno,oneside]{article}

\usepackage[utf8]{inputenc} %
\usepackage[T1]{fontenc}    %
\usepackage{microtype} 
\usepackage[english]{babel} %

\usepackage[a4paper,margin=1in,marginpar=1in]{geometry} %
\usepackage{csquotes}       %

\usepackage{authblk}  %

\usepackage[shortlabels]{enumitem}

\usepackage{amsmath}   %
\usepackage{amsthm}    %
\usepackage{mathtools} %
\usepackage[lcgreekalpha]{stix2}          %

\usepackage[
backend=biber,        %
style=ext-alphabetic,     %
giveninits=true,      %
maxalphanames=5,      %
minalphanames=5,      %
maxbibnames=99,       %
articlein=false,
]{biblatex}

\AtEveryBibitem{%
  \clearfield{note}%
  \clearlist{language}%
}

\DeclareSourcemap{
  \maps[datatype=bibtex]{
    \map{
      \step[fieldsource=doi,final]
      \step[fieldset=url,null]
    }  
  }
}

\newbibmacro*{bbx:parunit}{%
  \ifbibliography
    {\setunit{\bibpagerefpunct}\newblock
     \usebibmacro{pageref}%
     \clearlist{pageref}%
     \setunit{\adddot\par\nobreak}}
    {}}

\renewbibmacro*{doi+eprint+url}{%
  \usebibmacro{bbx:parunit}%
  \iftoggle{bbx:doi}
    {\printfield{doi}}
    {}%
  \iftoggle{bbx:eprint}
    {\usebibmacro{eprint}}
    {}%
  \iftoggle{bbx:url}
    {\usebibmacro{url+urldate}}
    {}}

\renewbibmacro*{eprint}{%
  \usebibmacro{bbx:parunit}%
  \iffieldundef{eprinttype}
    {\printfield{eprint}}
    {\printfield[eprint:\strfield{eprinttype}]{eprint}}}

\renewbibmacro*{url+urldate}{%
  \usebibmacro{bbx:parunit}%
  \printfield{url}%
  \iffieldundef{urlyear}
    {}
    {\setunit*{\addspace}%
     \printtext[urldate]{\printurldate}}}

\usepackage{xcolor}                   %
\definecolor{green}{HTML}{2ECC71}
\definecolor{blue}{HTML}{3498DB}
\definecolor{red}{HTML}{E74C3C}
\definecolor{orange}{HTML}{FD6A02}
\usepackage[pdfpagelabels]{hyperref}  %
\hypersetup{%
  colorlinks,
  linktocpage,
  urlcolor  = blue,
  citecolor = green,
  linkcolor = red}

\usepackage[capitalise,sort]{cleveref} %

\AfterEndEnvironment{proof}{\noindent\ignorespaces} %
\makeatletter
\def\@endtheorem{\endtrivlist}
\makeatother

\Crefname{paragraph}{\S}{\SS}
\Crefname{equation}{}{}
\Crefname{enumi}{}{}
\Crefname{conditioni}{Condition}{Conditions}
\Crefname{conditionalti}{Condition}{Conditions}
\newtheorem{theorem}{Theorem}[section]
\newtheorem*{theorem*}{Theorem}
\Crefname{theorem}{Theorem}{Theorems}
\newtheorem{theoremintro}{Theorem}
\Crefname{theoremintro}{Theorem}{Theorems}

\Crefname{theoremintro}{Assumption}{Assumptions}

\newtheorem{lemma}[theorem]{Lemma}
\Crefname{lemma}{Lemma}{Lemmas}
\newtheorem{proposition}[theorem]{Proposition}
\Crefname{proposition}{Proposition}{Propositions}
\newtheorem{corollary}[theorem]{Corollary}
\newtheorem*{corollary*}{Corollary}
\Crefname{corollary}{Corollary}{Corollaries}

\Crefname{conjecture}{Conjecture}{Conjectures}
\newtheorem{assumption}{Assumption}
\theoremstyle{definition}

\Crefname{example}{Example}{Examples}
\newtheorem*{example*}{Example}
\Crefname{assumption}{Assumption}{Assumptions}
\newtheorem{definition}[theorem]{Definition}
\Crefname{definition}{Definition}{Definitions}

\Crefname{question}{Question}{Questions}
\theoremstyle{remark}
\newtheorem{remark}[theorem]{Remark}
\Crefname{remark}{Remark}{Remarks}

\numberwithin{equation}{section} %

\usepackage{booktabs} %

\DeclarePairedDelimiter{\paren}{\lparen}{\rparen}
\DeclarePairedDelimiter{\bracket}{\lbrack}{\rbrack}
\DeclarePairedDelimiter{\set}{\lbrace}{\rbrace}
\DeclarePairedDelimiter{\abs}{\lvert}{\rvert}

\DeclarePairedDelimiter{\norm}{\lVert}{\rVert}
\DeclarePairedDelimiterX{\psh}[2]{\langle}{\rangle}{#1, #2}

\DeclarePairedDelimiterXPP{\Exp}[1]{\exp}{\lparen}{\rparen}{}{#1}
\DeclarePairedDelimiterXPP{\Log}[1]{\log}{\lparen}{\rparen}{}{#1}
\DeclarePairedDelimiterXPP{\Inf}[1]{\inf}{\lbrace}{\rbrace}{}{#1}
\DeclarePairedDelimiterXPP{\Sup}[1]{\sup}{\lbrace}{\rbrace}{}{#1}
\DeclarePairedDelimiterXPP{\Max}[1]{\max}{\lbrace}{\rbrace}{}{#1}
\DeclarePairedDelimiterXPP{\Min}[1]{\min}{\lbrace}{\rbrace}{}{#1}

\DeclareMathOperator{\tr}{Tr}
\DeclareMathOperator{\rk}{rank}
\DeclareMathOperator{\dom}{\mathbb{D}om}

\DeclareMathOperator{\esp}{\mathbf{E}}
\DeclareMathOperator{\prob}{\mathbf{P}}
\DeclareMathOperator{\var}{\mathbf{Var}}
\DeclareMathOperator{\law}{\mathbf{law}}

\DeclarePairedDelimiterXPP{\Prob}[1]{\prob}[]{}{#1}
\DeclarePairedDelimiterXPP{\Esp}[1]{\esp}[]{}{#1}
\DeclarePairedDelimiterXPP{\Var}[1]{\var}[]{}{#1}
\DeclarePairedDelimiterXPP{\Law}[1]{\law}[]{}{#1}

\author{Ronan HERRY}
\affil{%
  IRMAR, Université de Rennes 1
  \par
  \normalfont\href{mailto:ronan.herry@univ-rennes.fr}{\texttt{ronan.herry@univ-rennes.fr}}%
  \qquad ORCID: \href{https://orcid.org/0000-0001-6313-1372}{0000-0001-6313-1372}%
  }
\author{Dominique MALICET}
\affil{%
  LAMA, Université Gustave Eiffel
  \par
  \normalfont\href{mailto:dominique.malicet@univ-eiffel.fr}{\texttt{dominique.malicet@univ-eiffel.fr}}%
  \qquad ORCID: \href{https://orcid.org/0000-0003-2768-0125}{0000-0003-2768-0125}
  }
\author{Guillaume POLY}
\affil{%
  IRMAR, Université de Rennes 1
  \par
  \normalfont\href{mailto:guillaume.poly@univ-rennes.fr}{\texttt{guillaume.poly@univ-rennes.fr}}%
  }

\begin{document}
\title{Regularity of laws via Dirichlet forms -- Application to quadratic forms in independent and identically distributed random variables}
 \maketitle
 \vspace{-2em}
\begin{abstract}
  We study the regularity of the law of a quadratic form $Q(X,X)$, evaluated in a sequence $X = (X_{i})$ of independent and identically distributed random variables, when $X_{1}$ can be expressed as a sufficiently smooth function of a Gaussian field.
This setting encompasses a large class of important and frequently used distributions, such as, among others, Gaussian, Beta, for instance uniform, Gamma distributions, or else any polynomial transform of them.

Let us present an emblematic application.
Take $X = (X_{i})$ a sequence of independent and identically distributed centered random variables, with unit variance, following such distribution.
Consider also $(Q_{n})$ a sequence of quadratic forms, with associated symmetric Hilbert--Schmidt operators $(\mathsf{A}^{(n)})$.
Assume that $\tr\bracket*{ (\mathsf{A}^{(n)})^{2} } = 1/2$, $\mathsf{A}^{(n)}_{ii} =0$, and the spectral radius of $\mathsf{A}^{(n)}$ tends to $0$.
Then, $(Q_{n}(X))$ converges in a strong sense to the standard Gaussian distribution.
Namely, all derivatives of the densities, which are well-defined for $n$ sufficiently large, converge uniformly on $\mathbb{R}$ to the corresponding derivatives of the standard Gaussian density.

While classical methods, from Malliavin calculus or $\Gamma$-calculus, generally consist in bounding negative moments of the so-called \emph{carré du champ} operator $\Gamma(Q(X),Q(X))$, we provide a new paradigm through a second-order criterion involving the eigenvalues of a Hessian-type matrix related to $Q(X)$.
This Hessian is built by iterating twice a tailor-made gradient, the \emph{sharp operator} $\sharp$, obtained via a Gaussian representation of the carré du champ.
We believe that this method, recently developed by the authors in the current paper and \cite{HMP}, is of independent interest and could prove useful in other settings.
\end{abstract}

\section*{Introduction}

\subsection*{Main contributions}

\subsubsection*{Regularity of the law of random variables via Dirichlet form analysis}
We study the regularity of the law of a real-valued random variable $F$, that is a smooth function of a Gaussian field.
Thanks to the \emph{Malliavin calculus}, the underlying Gaussian structure, called the \emph{Wiener space}, comes equipped with a \emph{Dirichlet form}.
Here, we merely recall that the theory of Dirichlet form provides an ansatz of differential calculus at the level of the probability space, yielding, in particular, an integration by parts formalism.
Roughly speaking, our analysis builds upon the three following ingredients at the heart of the theory of Dirichlet forms:
\begin{enumerate}[(i)]
  \item the \emph{carré du champ}, an unbounded non-negative bilinear form $\Gamma$;
  \item the \emph{infinitesimal generator}, an unbounded linear operator $\mathsf{L}$;
  \item and an algebra of \emph{smooth} random variables $\mathbb{D}^{\infty}$, stable under $\mathsf{L}$ and $\Gamma$.
\end{enumerate}
Reminders on Dirichlet forms on the Wiener space are given in \cref{s:dirichlet-wiener}.
We nevertheless recall, in this introduction, the fundamental interplay between those three objects through the integration by parts formula:
\begin{equation*}
  \Esp*{ -X\mathsf{L}Y} = \Esp*{ \Gamma(X,Y) }, \qquad X,\, Y \in \mathbb{D}^{\infty}.
\end{equation*}
Our techniques allow us to study finely the regularity of the law of a quadratic form in smooth random variables.
To formulate our main result, let us define, for a symmetric Hilbert--Schmidt operator $\mathsf{A}$ with spectrum $(\lambda_{i})$, the \emph{spectral remainders}
\begin{equation*}
  \mathcal{R}_{q}(\mathsf{A}) \coloneq \sum_{i_{1} \ne \dots \ne i_{q}} \lambda^{2}_{i_{1}} \dots \lambda^{2}_{i_{q}}, \qquad q \in \mathbb{N},
\end{equation*}
We also recall the definition of the \emph{influence} of $\mathsf{A}$
\begin{equation*}
  \tau(\mathsf{A}) \coloneq \sup_{i \in \mathbb{N}}  \sum_{j \in \mathbb{N}} \paren*{a_{ij}^{(n)}}^{2}.
\end{equation*}

\begin{theoremintro}[{\cref{th:regularity-quadratic-form}}]\label{th:intro-regularity-quadratic-form}
  Let $(X_{i})$ be independent copies of a centred and unit variance random variable $X \in \mathbb{D}^{\infty}$.
  Assume, moreover that,
  \begin{equation}\label{eq:intro-small-ball-gamma}
    \Esp*{ \Gamma(X, X)^{-\theta} } < +\infty,
  \end{equation}
  for some $\theta > 0$.
  Let $\set*{\mathsf{A}^{(n)} = (a_{ij}^{(n)}) : n \in \mathbb{N} }$ be a sequence of symmetric Hilbert--Schmidt operators on $\ell^{2}(\mathbb{N})$ such that $a_{ii}^{(n)} = 0$ for $i \in \mathbb{N}$, and
  \begin{align*}
    & \liminf_{n \to \infty} \mathcal{R}_{q}(\mathsf{A}^{(n)}) > 0, \qquad q \in \mathbb{N};
  \\&\lim_{n \to \infty} \tau(\mathsf{A}^{(n)}) = 0.
  \end{align*}
  Then, for all $q \in \mathbb{N}$, there exists $N_{q}$ such that
  \begin{equation*}
    \Law*{\psh{X}{\mathsf{A}^{(n)} X}} \in \mathscr{C}^{q},\qquad n \geq N_{q}.
  \end{equation*}
  Moreover,
  \begin{equation*}
    \sup_{n \geq N_{q}} \norm*{\Law*{\psh{X}{\mathsf{A}^{(n)}X}}}_{\mathscr{C}^{q}} < +\infty.
  \end{equation*}
\end{theoremintro}

Let us illustrate \cref{th:intro-regularity-quadratic-form} by a concrete application.
To our knowledge, the following corollary is the first result in the literature providing $\mathscr{C}^{\infty}$-convergence for non-linear polynomial functionals of non-Gaussian random variables.
For the sake of simplicity, we state the result in the setting of uniform distributions.
As established in \cref{s:examples}, the uniform variables $(U_{i})$ in the Corollary below can be taken in a larger class of random variables, including, for instance, the Gaussian variables, the Beta variables, the Gamma variables, and the multi-linear polynomials thereof.
We say that a sequence of random variables $(X_{n})$ converges $\mathscr{C}^{\infty}$ to a random variable $X$ with smooth density $f$ provided for all $p \in \mathbb{N}$, there exists $N_{p}$ such that for $n \geq N_{p}$, $X_{n}$ admits a density $f_{n} \in \mathscr{C}^{p}$ and
\begin{equation*}
  \norm{f_{n}^{(p)} - f^{(p)}}_{\infty} \to 0, \qquad p \in \mathbb{N}.
\end{equation*}
In this case, we write $X_{n} \xrightarrow[n \to \infty]{\mathscr{C}^{\infty}} X$.
\begin{corollary*} 
  Take $\set*{\mathsf{A}^{(n)} = (a_{ij}^{(n)}) : n \in \mathbb{N} }$ a sequence of symmetric Hilbert--Schmidt operators on $\ell^{2}(\mathbb{N})$ such that $a_{ii}^{(n)} = 0$ for $i \in \mathbb{N}$, and
\begin{align*}
  & \tr\bracket*{ (\mathsf{A}^{(n)})^{2} } = 1,
\\& \rho(\mathsf{A}^{(n)}) \coloneq \sup_{\lambda \in \mathrm{spec}(\mathsf{A}^{(n)})} \abs{\lambda} \xrightarrow[n \to \infty]{} 0.
\end{align*}
Take also $(U_{i})$ a sequence of independent uniform variables on $[-1,1]$.
Then
\begin{equation*}
  \sum_{ij} a_{ij}^{(n)} U_{i} U_{j} \xrightarrow[n \to \infty]{\mathscr{C}^{\infty}} \mathcal{N}(0,2/3).
\end{equation*}
\end{corollary*}

\begin{proof}
  By a classical result of \cite{Rotar}, the above spectral conditions imply convergence in law to a Gaussian.
  Moreover, by \cref{th:examples:beta} the uniform distribution satisfies the assumption of \cref{th:intro-regularity-quadratic-form}.
  In particular, all the $\mathscr{C}^{q}$-norms of the cumulative distribution function $F_{n}$ are uniformly bounded (for $n$ large enough).
Since the Gaussian distribution has no atom, by a standard Dini-type argument, $(F_{n})$ converges uniformly.
By Landau--Kolmogorov inequality
\begin{equation*}
  \norm{F_{n}}_{\mathscr{C}^{q}} \leq \sqrt{2} \norm{F_{n}}_{\mathscr{C}^{q-1}}^{1/2} \norm{F_{n}}_{\mathscr{C}^{q+1}}^{1/2},
\end{equation*}
and, thus, proceeding by induction, $(F_{n})$ converges in all $\mathscr{C}^{q}$.
\end{proof}

\subsection*{Motivations and related works}

\subsubsection*{Smooth central limit theorem}
It is part of the probabilistic folklore, that if $(X_{i})$ is a sequence of centred, normalized, independent and identically distributed random variables whose common law has density in a Sobolev space, then the regularity of the law of the linear functionals
\begin{equation*}
  S_{n} \coloneq \frac{1}{n^{1/2}} \sum_{i=1}^{n} X_{i},
\end{equation*}
improves with $n$.
This fact can be seen at the level of the decay of the Fourier transform of $S_{n}$, by exploiting the linearity of $S_{n}$ and the associated convolution structure, see for instance \cite{LionsToscani}.
Our work leverages the theory of Malliavin calculus on the Wiener space to obtain a non-linear equivalent of this regularization phenomenon.

Improving the type of convergence in the central limit theorem for non-linear functionals of random fields is a longstanding and intensively studied problem in probability theory.
The lack of linearity rules out methods based on convolution structure and its regularization properties.
Hence, establishing smooth limit theorems, in this non-linear setting, turns out to be a much harder task.

When the $X_{i}$'s are independent Gaussian, let us mention, among other, the works \cite{NourdinPolyTotalVariation,HuLuNualart,BallyCarmellinoWiener} that establish that central convergence in law can be improved to convergence in total variation; while \cite{NPY} derives similar result for convergence in entropy.
In our companion paper \cite{HMP}, we derive that, on Wiener chaoses, central convergence can actually always be improved to $\mathscr{C}^{\infty}$ convergence of the densities.

For non-Gaussian random variables, less results are available.
See however, \cite{BallyCarmellinoInvariance,BallyCaramellinoPoly} where convergence in total variation is considered when the common law of uderlyings $X_{i}$'s admits an absolutely continuous part.
As in the present article, all the aforementioned results are derived through ideas pertaining to Malliavin calculus and Dirichlet forms.

We stress out that we do not need to assume asymptotic normality, or even convergence in law, of the random variables under consideration, as our conditions in \cref{th:intro-regularity-quadratic-form} hold in a wider context.
In some cases, it is however, possible to recover those conditions from asymptotic normality.
See for instance \cref{s:normal-convergence} for concrete examples.

\subsubsection*{Regularity of the law via positivity of the carré du champ}

In his seminal work \citeauthor{Malliavin} \cite{Malliavin} lays the foundation of an infinite dimensional differential calculus at the level of the Wiener space to study the regularity of the density of a solution of a SDE, thus giving a new proof of a celebrated theorem by \citeauthor{HormanderHypoelliptic} \cite{HormanderHypoelliptic}.
\cite{Malliavin} establishes that some form of positivity of the Malliavin gradient implies some form of regularity at the level of the law.
Since then, the existence of negative moments of the Malliavin gradient plays a prominent role in many works related to Gaussian analysis (see, for instance, \cite{CassFriz,HNTXBreuerMAjor,NourdinNualartFisher,AruGMC,Schoenbauer} for recent works).
In the literature, the problem of existence of negative moments is, most of the time, tackled on a case by case basis exploiting the specificity of the model under consideration, or taken as an assumption.
To the best of our knowledge, \cite{BogachevKosovZelenov} is the only work developing a comprehensive theory of regularity for general Gaussian polynomials.
They obtain results regarding Besov regularity, and their work is also based on the study of the Malliavin derivative.
In this work, we introduce novel ideas to systematically study the existence of negative moments of the Malliavin derivative, allowing us to derive $\mathscr{C}^{\infty}$ regularity of the density beyond the case of Gaussian polynomials.

\subsection*{Outline of the proof and of our construction}

\subsubsection*{A Gaussian representation of the carré du champ}

As anticipated, we derive explicit controls of the negative moments of $\Gamma(F,F)$, for some $F = F(X_{1}, X_{2}, \dots) \in \mathbb{D}^{\infty}$, where $(X_{i})$ are independent copies of $X \in \mathbb{D}^{\infty}$ satisfying $\Esp*{ \Gamma(X,X)^{-\theta}} < \infty$ for some $\theta > 0$.
Note that each $X_{i}$ is a function of variables $(Y_{i,k})_{k}$ such that $(Y_{i,k})_{i,k}$ is an array of independent Gaussian variables.
We define \cref{s:bouleau-derivative} the \emph{Bouleau derivative}:
\begin{equation*}
  \sharp_{G} F \coloneq \sum_{i \in \mathbb{N}} (\partial_{x_{i}}F) (\mathsf{D} X_{i} \cdot \vec{G}_{i}),
\end{equation*}
where $G = (G_{i,k})_{i,k}$ is an array of independent standard Gaussian variables independent of the underlying Wiener space generated by the $(Y_{ik})_{ik}$, $\vec{G}_{i} \coloneq (G_{i,k})_{k}$, and $\mathsf{D} X_{i} \coloneq (\partial_{y_{k}} X_{i})_{k}$.
By construction, $\sharp_{G}F$ has the same law as $\Gamma(F,F)^{1/2}N$ where $N$ is an independent standard Gaussian variable.
In particular, we have the following \emph{Fourier--Laplace identity}:
\begin{equation*}
  \Esp*{ \Exp*{\mathrm{i} t \sharp_{G}F} } = \Esp*{ \Exp*{-\frac{t^{2}}{2} \Gamma(F,F)} }.
\end{equation*}
From this identity, we see that existence of negative moments for $\Gamma(F,F)$ is equivalent to some sufficiently fast Fourier decay of $\sharp_{G}F$, which in turn is equivalent to some regularity of the law of $\sharp_{G}F$.
Combining all of this, we obtain the following intermediary result.
\begin{theoremintro}[{\cref{th:sobolev-regularity-derivative}}]\label{th:intro-sobolev-regularity-derivative}
  Assume that $\sharp_{G}F$ has a smooth law, then so does $F$.
\end{theoremintro}

\begin{remark}
  We have used the Bouleau derivative and a similar idea in \cite{HMP} to study regularization on Wiener chaoses.
\end{remark}

\subsubsection*{Comparison of the Malliavin derivative and the Bouleau derivative}
In the theory of Dirichlet forms, and in particular in Gaussian analysis, it is rather standard to represent the carré du champ $\Gamma$ through a so-called \emph{Malliavin derivative}.
This consists in a separable Hilbert space $\mathbb{G}$ together with a map $\nabla \colon L^{2} \to L^{2}(\mathbb{G})$ such that
\begin{equation*}
  \norm{\nabla F}_{\mathbb{G}}^{2} = \Gamma(F,F).
\end{equation*}
All separable Hilbert spaces being isomorphic little care is usually given to the choice of $\mathbb{G}$.
The Bouleau derivative $\sharp_{G}F$ is a Malliavin derivative, where we take $\mathbb{G}$ to specifically be a Gaussian space.
By doing so, we introduce gaussianity in a \emph{a priori} non-Gaussian world, and we can now leverage the rich structure of Gaussian analysis.
This paradigmatic shift in the representation of the derivative is reminiscent of proof of the isometric embeddality of $\ell^{2}(\mathbb{N})$ into $L^{q}(0,1)$, where $\ell^{2}(\mathbb{N})$ is explicitly sent into the Gaussian space.
Namely, we identify every $v = (v_{i}) \in \ell^{2}(\mathbb{N})$ with a Gaussian variable $G(v) \coloneq \sum_{i} v_{i} N_{i}$ with $(N_{i})$ a standard Gaussian sequence.
Then, $G(v)$ has finite $q$-moment for all $q \in \mathbb{N}$, an information that we could have gained, have we stayed in $\ell^{2}(\mathbb{N})$.
The idea of taking derivatives in the direction of Gaussian variables can be traced back to a slightly different construction of \citeauthor{BouleauError} \cite[Chap.\ V \S 2]{BouleauError} to study different problems.

\subsubsection*{Introducing a non-linearity by iterating the derivative}
The object $\sharp_{G}F$ is still too close from $F$, and we were not able to derive meaningful estimates by working directly at the level of $\sharp_{G}F$.
This motivates the introduction of the second derivative $\sharp_{H} \sharp_{G}F$.
For another array $H = (H_{j,l})_{jl}$ of independent standard Gaussian variables, we define \emph{iterated Bouleau derivative}
\begin{equation*}
 \sharp_{H} \sharp_{G} F \coloneq \sum_{i,j \in \mathbb{N}} (\partial_{x_{i}} \partial_{x_{j}}F) (\mathsf{D} X_{i} \cdot \vec{G}_{i}) (\mathsf{D} X_{j} \cdot \vec{H}_{j}) + \sum_{i \in \mathbb{N}} (\partial_{x_{i}} F) \vec{H}_{i} \cdot (\mathsf{D}^{2}X_{i} \vec{G}_{i}),
\end{equation*}
where $\mathsf{D}^{2}X_{i} \coloneq (\partial_{y_{k}}\partial_{y_{l}}X_{i})_{kl}$.
This definition simply corresponds to applying the Bouleau derivative defined above to the random variable $\sharp_{G}F$ where in this case we see $G$ as being frozen.
The first sum on the right-hand side has the same law as $\psh{U}{(\Gamma^{1/2}(\nabla^{2}F)\Gamma^{1/2}) V}$, where $(\nabla^{2}F)_{ij} = (\partial_{x_{i}}\partial_{x_{j}}F)$, $\Gamma \coloneq \mathrm{diag}(\Gamma(X_{1}, X_{1}), \Gamma(X_{2}, X_{2}), \dots)$, and $U = (U_{i})$ and $V = (V_{i})$ are two independent standard Gaussian vectors.
Similarly to T.\ Royen's proof of the Gaussian correlation inequality \cite{Royen}, we exploit properties of square of Gaussian variables rather than of Gaussian variables directly.
In our case, this allows us to control, in \cref{th:regularity-negative-moments-spectral-remainder}, the regularity of $\sharp_{H} \sharp_{G} F$ in terms of a quantity related to $\Gamma^{1/2} (\nabla^{2}F) \Gamma^{1/2}$.

Iterating the argument leading to \cref{th:intro-sobolev-regularity-derivative}, we obtain in \cref{th:sobolev-regularity-iterated-gradient}, that whenever the law of $\sharp_{H} \sharp_{G} F$ is smooth then the law of $F$ is also smooth.
Concretely our reasoning rely on the following chain of implications
\begin{equation*}
  \begin{split}
  & \sharp_{H} \sharp_{G} F \ \text{smooth} \overset{\text{Fourier--Laplace}}{\Longrightarrow} \Esp*{\Gamma(\sharp_{G}F, \sharp_{G}F)^{-q}} < \infty, \forall q \in \mathbb{N} \overset{\text{Malliavin calculus}}{\Longrightarrow} \sharp_{G}F \ \text{smooth} 
\\& \overset{\text{Fourier--Laplace}}{\Longrightarrow} \Esp*{\Gamma(F,F)^{-q}} < \infty, \forall q \in \mathbb{N} \overset{\text{Malliavin calculus}}{\Longrightarrow} F \ \text{smooth}.
  \end{split}
\end{equation*}

\subsubsection*{Controlling the iterated Bouleau derivative of a quadratic form}
When considering a quadratic form $F \coloneq \psh{X}{\mathsf{A}X}$, with $\mathsf{A}$ a deterministic symmetric Hilbert--Schmidt operators, the term to control in the iterated Bouleau derivative $\sharp_{H} \sharp_{G} F$ is of the special form: $\Gamma^{1/2} \nabla^{2}F \Gamma^{1/2} = \Gamma^{1/2} \mathsf{A} \Gamma^{1/2}$, where the randomness only comes from the diagonal matrix $\Gamma$.
Owing to this particular form, we can explicitly control this quantity under some spectral assumptions on $\mathsf{A}$, which gives \cref{th:intro-regularity-quadratic-form}.
As a main technical tool, we use two results of independent interest:
\begin{itemize}[wide]
  \item a result showing that existence of a negative moment is preserved by taking multi-linear polynomials (\cref{th:small-ball-l2});
  \item a splitting argument that allows us to improve the positivity of some quantities under spectral conditions (\cref{th:small-ball-improved}).
\end{itemize}
These results complement, and are actually based, on existing results regarding invariance and anti-concentration for multi-linear polynomials \cite{CarberyWright,MDOInvariance}.

\subsubsection*{Embedding common laws in our setting}
To apply our result, we verify that many common laws fit in our setting.
More precisely, in \cref{th:examples:general}, we show that under polynomial growth conditions on the derivatives of the quantile function of a law $\mu$, there exists $X \in \mathbb{D}^{\infty}$ with law $\mu$.
The condition $\Esp*{ \Gamma(X,X)^{-\theta}} < \infty$, is satisfied as soon as $\mu$ admits a density in $L^{1+\varepsilon}$ for some $\varepsilon > 0$.
In particular, we verify in \cref{th:examples:beta,th:examples:gamma} that all the Beta and Gamma distributions can be realized on the Wiener space in a way that is suitable to apply our result.
Even though the proofs of the above theorems are mostly computational, the authors were not aware that the Gaussian setting was rich enough to support such a variety of laws, and we find this phenomenon of independent interest.

\subsubsection*{Acknowledgements}
We heartily thank the anonymous referee for their thorough review and detailed comments that have considerably helped us to improve the paper.

\subsubsection*{Funding}
R.H.~gratefully acknowledges funding from Centre Henri Lebesgue (ANR-11-LABX-0020-01) through a research fellowship in the framework of the France 2030 program.
This work was supported by the ANR Grant UNIRANDOM (ANR-17-CE40-0008).
An important part of this work was done while D.M.~was visiting the IRMAR in Rennes thanks to a CNRS \emph{délégation}.

\tableofcontents%

\setcounter{secnumdepth}{5}

\section{Reminders and notations}

\subsection{Linear algebra}\label{s:linear-algebra}

\subsubsection{Symmetric Hilbert--Schmidt operators}
We write $\ell^{2}(\mathbb{N})$ for the usual Hilbert space of square-integrable $\mathbb{N}$-indexed sequences $x = (x_{i} : i \in \mathbb{N})$.
A \emph{symmetric Hilbert--Schmidt operator} $\mathsf{A}$ acting on $\ell^{2}(\mathbb{N})$ can be identified with a bi-sequence $\mathsf{A} \simeq (a_{ij} : i,\, j \in \mathbb{N})$ satisfying:
\begin{description}
  \item[symmetry] $a_{ij} = a_{ji}$ for all $i$ and $j \in \mathbb{N}$;
  \item[square-integrability] $\tr \mathsf{A}^{2} \coloneq \norm{a}_{\ell^{2}}^{2} \coloneq \sum_{i,j} a_{ij}^{2} < \infty$;
\end{description}
We sometimes also assume that $\mathsf{A}$ has a \emph{vanishing diagonal}, that is $a_{ii} = 0$ for all $i \in \mathbb{N}$.

\paragraph{Eigenvalues and spectral quantities}
The operator $\mathsf{A}$ is diagonalizable in an orthonormal basis with real eigenvalues $(\lambda_{i} : i \in \mathbb{N})$.
We always assume that they are decreasingly ordered by the absolute values of their eigenvalues: $\abs{\lambda_{1}} \geq \abs{\lambda_{2}} \geq \dots$.
The square-integrability can be rephrased in terms of the eigenvalues
\begin{equation*}
  \tr \mathsf{A}^{2} = \sum_{i \in \mathbb{N}} \lambda_{i}^{2} < \infty.
\end{equation*}
We consider various quantities associated with $\mathsf{A}$:
\begin{itemize}
  \item the \emph{spectral radius} 
\begin{equation*}
  \rho(\mathsf{A}) \coloneq \sup_{i \in \mathbb{N}} \abs{\lambda_{i}};
\end{equation*}
\item the \emph{spectral remainders}
\begin{equation*}
  \mathcal{R}_{q}(\mathsf{A}) \coloneq \sum_{i_{1} \ne \dots \ne i_{q}} \lambda_{i_{1}}^{2} \dots \lambda_{i_{q}}^{2}, \qquad q \in \mathbb{N};
\end{equation*}
\item the \emph{partial influences}
\begin{equation*}
  \tau_{i}(\mathsf{A}) \coloneq \sum_{j \in \mathbb{N}} a_{ij}^{2}, \qquad i \in \mathbb{N};
\end{equation*}
\item the \emph{maximal influence}
\begin{equation*}
  \tau(\mathsf{A}) \coloneq \sup_{i \in \mathbb{N}} \tau_{i}(\mathsf{A}) = \sup_{i \in \mathbb{N}} \sum_{j \in \mathbb{N}} a_{ij}^{2}.
\end{equation*}
\end{itemize}

We always have that
\begin{equation*}
  \tau_{i}(\mathsf{A}) \leq \tau(\mathsf{A}) \leq \tr \mathsf{A}^{2}.
\end{equation*}
For $I$ and $J \subset \mathbb{N}$, we also write
\begin{equation*}
  \mathsf{A}(I,J) \coloneq (a_{ij} : (i,j) \in I \times J),
\end{equation*}
for the extracted operator.
As anticipated, our \cref{th:regularity-quadratic-form} deals with operator whose spectral remainders are positive and influence is small.
The following result shows that it is actually sufficient to control the spectral radius.

\begin{lemma}\label{th:spectral-radius-implies-influence}
  Let $\mathsf{A}$ be a symmetric Hilbert--Schmidt operator with $\tr \mathsf{A}^{2} = 1$.
  Then,
  \begin{align}
    & \label{eq:bound-spectral-remainder-spectral-radius} \mathcal{R}_{q}(\mathsf{A}) \geq \prod_{k=1}^{q-1} (1- k \rho(\mathsf{A})^{2})
  \\& \label{eq:bound-influence-spectral-radius} \tau(\mathsf{A}) \leq \rho(\mathsf{A})^{2}.
  \end{align}
\end{lemma}

\begin{proof}
  By definition,
  \begin{equation*}
    \mathcal{R}_{q}(\mathsf{A}) = \sum_{i_{1} \ne \dots \ne i_{q-1}} \lambda_{i_{1}}^{2} \dots \lambda_{i_{q-1}}^{2} \paren*{ \sum_{i_{q} \not\in \set{i_{1},\dots, i_{q-1}}} \lambda_{i_{q}}^{2} } \geq \mathcal{R}_{q-1}(\mathsf{A}) \sum_{i \geq q} \lambda_{i}^{2} \geq \mathcal{R}_{q-1}(\mathsf{A}) (1-(q-1)\rho(\mathsf{A})^{2}).
  \end{equation*}
  The first inequality is true since the $(\lambda_{i}^{2})$ is non-increasing. 
  \cref{eq:bound-spectral-remainder-spectral-radius} follows by an immediate induction.
Take $i \in \mathbb{N}$.
For \cref{eq:bound-influence-spectral-radius}, consider $e_{i}$ the $i$-th vector of the canonical basis.
Then,
\begin{equation*}
  \tau_{i}(\mathsf{A}) = \norm{\mathsf{A} e_{i}}^{2} \leq \rho(\mathsf{A})^{2}.
\end{equation*}
Taking the supremum over $i \in \mathbb{N}$ completes the proof.
\end{proof}

\paragraph{Properties of the spectral remainders}
Provided $\rk \mathsf{A} < q$, then $\mathcal{R}_{q}(\mathsf{A}) = 0$.
Actually, the spectral remainder $\mathcal{R}_{q}(\mathsf{A})$ measures the distance of $\mathsf{A}$ to the operators of rank $q$.
See \cite[Lem.\ 3 \& 4]{HMP} for a precise statement.

The following representation of the spectral remainder will come handy.
\begin{lemma}[{\cite[Thm.\ 6]{CauchyBinet}}]\label{th:cauchy-binet}
  For a symmetric Hilbert--Schmidt operator $\mathsf{A}$:
  \begin{equation*}
    \mathcal{R}_{q}(\mathsf{A}) = \sum_{\substack{I,\, J \subset \mathbb{N}\\\abs{I} = \abs{J} = q}} \bracket*{\det \mathsf{A}(I,J)}^{2}.
  \end{equation*}
\end{lemma}

\subsubsection{Determinantal operators associated with a Hilbert--Schmidt operator}\label{s:linear-algebra-ell-1}

Let $q \in \mathbb{N}$ and $\mathscr{P}_{q}(\mathbb{N}) \coloneq \set*{ I \subset \mathbb{N} : \abs{I} = q }$.
We identify $\mathscr{P}_{1}(\mathbb{N}) = \mathbb{N}$.
In view of the Cauchy--Binet formula, it is natural to consider the operator $\mathsf{B} \coloneq (b_{IJ} : I,\, J \in \mathscr{P}_{q}(\mathbb{N}))$ with non-negative coefficients
\begin{equation*}
  b_{IJ} \coloneq \bracket*{\det \mathsf{A}(I,J)}^{2}.
\end{equation*}
We treat $\mathsf{B}$ as a genuine operator, however due to the non-negative nature of its entries, we consider $\ell^{1}$-quantities rather than $\ell^{2}$:
\begin{itemize}
  \item  the \emph{total mass}:
    \begin{equation*}
      \sigma(\mathsf{B}) \coloneq \norm{b}_{\ell^{1}} \coloneq \sum_{I,J} b_{IJ} = \mathcal{R}_{q}(\mathsf{A});
    \end{equation*}
  \item the \emph{$\ell^{1}$-partial influences} of the index $i$:
    \begin{equation*}
      \upsilon_{i}(\mathsf{B}) \coloneq \sum_{I,\,J \in \mathscr{P}_{q}(\mathbb{N})} 1_{i \in I \cup J} b_{IJ};
    \end{equation*}
  \item the \emph{$\ell^{1}$-maximal influences}:
    \begin{equation*}
      \upsilon(\mathsf{B}) \coloneq \sup_{i} \upsilon_{i}(\mathsf{B}).
    \end{equation*}
\end{itemize}

\begin{lemma}\label{th:influence-det}
  With the above notations, assume that $\tr \mathsf{A}^{2} = 1$, then
  \begin{equation*}
    \upsilon(\mathsf{B}) \leq 2 q \tau(\mathsf{A}).
  \end{equation*}
\end{lemma}

\begin{proof}
  Let $i \in \mathbb{N}$.
  For $I$ and $J\in \mathscr{P}_{q}(\mathbb{N})$, write $I = \set{ i_{1}, \dots, i_{q}}$ and $J = \set{j_{1}, \dots, j_{q}}$ with the elements being increasingly ordered.
  By definition, $\Sigma_{q}$ the set of permutations on $\{1,\dots, q\}$, we have
  \begin{equation*}
    \det \mathsf{A}(I, J) = \sum_{\sigma \in \Sigma_{q}} (-1)^{\abs{\sigma}} \prod_{l=1}^{q} a_{i_{l} j_{\sigma(l)}}.
  \end{equation*}
  By Jensen's inequality,
  \begin{equation*}
    b_{IJ} = \bracket*{ \det \mathsf{A}(I,J) }^{2} \leq q! \sum_{\sigma \in \Sigma_{q}} \prod_{l=1}^{q} a_{i_{l}j_{\sigma(l)}}^{2}.
  \end{equation*}
  Thus, we find that
  \begin{equation*}
    \upsilon_{i}(\mathsf{B}) \leq \sum_{I \cup J \ni i} q! \sum_{\sigma \in \Sigma_{q}} \prod_{l=1}^{q} a^{2}_{i_{l}j_{\sigma(l)}}.
  \end{equation*}
  On the one hand, when $(j_{1}, \dots, j_{q})$ ranges through increasingly ordered sets, $(j_{\sigma(1)}, \dots, j_{\sigma(q)})$ for $\sigma$ ranging in $\Sigma_{q}$ ranges through all non-ordered sets.
  On the other hand to go from increasingly ordered $(i_{1}, \dots, i_{q})$ to pairwise disjoint $(i_{1}, \dots, i_{q})$, we have to pay a factor $q!$.
  It follows that
  \begin{equation*}
    \upsilon_{i}(\mathsf{B}) \leq 2 \sum_{i_{1} \ne \dots \ne i_{q}} \sum_{j_{1} \ne \dots \ne j_{q}} 1_{\set*{\exists l : i_{l} = i}} \prod_{l=1}^{q} a^{2}_{i_{l}j_{l}}.
  \end{equation*}
  By symmetry, we finally get
  \begin{equation*}
    \upsilon_{i}(\mathsf{B}) \leq 2 q \paren*{\sum_{j} a^{2}_{ij}}\paren*{ \sum_{k,j} a_{kj}^{2}}^{q-1} = 2q \tau_{i}(\mathsf{A}) \paren*{\tr \mathsf{A}^{2}}^{q-1}.
  \end{equation*}
  This concludes the proof since $\tr \mathsf{A}^{2} = 1$.
\end{proof}

\subsection{Sobolev regularity and Sobolev regularity in Fourier modes}

\subsubsection{Regularity for functions}
We recall notions regarding Sobolev regularity of a function $f$ that we use to measure the regularity of the density of random variables.

\paragraph{Hölder regularity.}
For $k \in \mathbb{N}$, we write $\mathscr{C}^{k}$ for the space of $k$ times continuously differentiable functions on $\mathbb{R}$, equipped with the norm
\begin{equation*}
  \norm{f}_{\mathscr{C}^{k}} \coloneq \max_{l=1,\dots,k} \Sup*{ \abs{f^{(l)}(x)} : x \in \mathbb{R} }.
\end{equation*}
Additionally, for $\alpha \in [0,1]$, we write $\mathscr{C}^{k,\alpha}$ for the space of $f \in \mathscr{C}^{k}$ whose $k$-th derivative is also $\alpha$-Hölder, equipped with the norm
\begin{equation*}
  \norm{f}_{\mathscr{C}^{k,\alpha}} \coloneq \norm{f}_{\mathscr{C}^{k}} + \sup_{x\ne y} \frac{\abs{f^{(k)}(x) - f^{(k)}(y)}}{\abs{x-y}^{\alpha}}.
\end{equation*}

\paragraph{Sobolev regularity.}
Actually, it is more convenient to work with the notion of weak regularity whose basic definitions are recalled below.
For $p \in [1,\infty]$, we write $\mathscr{L}^{p}$ for the \emph{Lebesgue space of order $p$} on $\mathbb{R}$.
We also define $\mathscr{S}$ is the \emph{Schwartz space} of rapidly decreasing functions on $\mathbb{R}$, and $\mathscr{S}'$ is its dual the space of \emph{tempered distributions}.
On $\mathscr{S}'$, we can define by duality a \emph{derivative operator} $\partial$.
By spectral calculus, we can actually make sense of any power of the Laplace operator, that is we can define, for all $s \in \mathbb{R}_{+}$, the operator $(1-\partial^{2})^{s}$ on $\mathscr{S}'$.
Every element $f \in \mathscr{L}^{p}$ induces a tempered distribution still denoted by $f$.
Conversely, for a tempered distribution $T$ we write $T \in \mathscr{L}^{p}$ whenever it is induced by a function in $\mathscr{L}^{p}$.
Recall the definition of the \emph{fractional Sobolev spaces}:
\begin{equation*}
  \mathscr{W}^{s,p} \coloneq \set*{ f \in \mathscr{S}' : (1-\partial^{2})^{s/2} f \in \mathscr{L}^{p} },  \qquad p \in [1,\infty],\, s \in \mathbb{R},
\end{equation*}
equipped with the norm
\begin{equation*}
  \norm{f}_{\mathscr{W}^{s,p}} \coloneq \norm{(1 - \partial^{2})^{s/2} f}_{\mathscr{L}^{p}}.
\end{equation*}
For $s \in \mathbb{N}$, the Sobolev norm $\norm{\cdot}_{\mathscr{W}^{s,p}}$ is equivalent to the \emph{classical} Sobolev norm
\begin{equation*}
  \norm{f}_{\mathscr{L}^{p}} + \norm{\partial^{s}f}_{\mathscr{L}^{p}}.
\end{equation*}
Sobolev spaces are relevant for our analysis due to the celebrated \emph{Sobolev embeddings}
\begin{equation}\label{eq:embedding-sobolev-holder}
  \mathscr{W}^{s,p} \hookrightarrow \mathscr{C}^{k,\alpha}, \qquad k + \alpha = s - \frac{1}{p},\, k \in \mathbb{N}, \, \alpha \in [0,1].
\end{equation}

\paragraph{Sobolev regularity in Fourier mode.}
In this work, it is convenient to work with the \emph{Fourier transform}.
Recall that by duality, we can define, on $\mathscr{S}'$, the \emph{Fourier transform}, and its inverse $\mathscr{F}^{-1}$.
Recall that the Fourier transform sends differential operators to multiplication operators.
In particular,
\begin{equation*}
  \mathcal{F}\bracket*{(1-\partial^{2})^{s} f}(\xi) = (1+\xi^{2})^{s} \hat{f}(\xi), \qquad \xi \in \mathbb{R},\, s \in \mathbb{R}.
\end{equation*}
We also write $\hat{f} = \mathcal{F} f$.
We then define the \emph{Sobolev space in Fourier mode}
\begin{equation*}
  \mathcal{F} \mathscr{W}^{s,p} \coloneq \set*{ f \in \mathscr{S}' : ( \xi \mapsto (1+\xi^{2})^{s/2} \hat{f}(\xi) ) \in \mathscr{L}^{p} }, \qquad p \in [1,\infty], \, s \in \mathbb{R},
\end{equation*}
equipped with the norm
\begin{equation*}
  \norm{f}_{\mathcal{F}\mathscr{W}^{s,p}} \coloneq \norm{(1+\xi^{2})^{s/2} \hat{f}}_{\mathscr{L}^{p}} = \norm*{ \mathcal{F} \bracket*{ (1-\partial^{2})^{s/2} f} }_{\mathscr{L}^{p}}.
\end{equation*}
Due to the Fourier isomorphism theorem \cite[Thm.\ 7.1.11]{Hormander}, we have that $\mathcal{F} \mathscr{W}^{s,2} = \mathscr{W}^{s,2}$ for all $s \in \mathbb{R}$.
In general, by \cite[Thm.\ 7.1.13]{Hormander}, for $p \in [1,2]$ and $p' \coloneq p/(p-1) \in [2,\infty]$ the Hölder conjugate of $p$, we have that $\mathcal{F} \colon \mathscr{L}^{p} \to \mathscr{L}^{p'}$ is bounded.
While, by \cite[Thm.\ 7.9.3]{Hormander}, for $q > 2$, and $s > 1/2 - 1/q$, $\mathcal{F} \colon \mathscr{L}^{q} \to \mathscr{W}^{-s,2}$ is bounded.
Consequently, we have that
\begin{equation}\label{eq:embedding-fourier-sobolev}
  \mathscr{W}^{s,p} \hookrightarrow \mathcal{F} \mathscr{W}^{s,p'} \hookrightarrow \mathscr{W}^{s-t,2}, \qquad p \in [1,2],\, s\in \mathbb{R},\, t > \frac{1}{p} - \frac{1}{2}.
\end{equation}

\subsubsection{Regularity at the level of the random variable}

It is more convenient to manipulate quantities defined at the level of a random variables rather than referring to its density.
Whenever a random variable $Z$ has density $f$, we let
\begin{align*}
  & \mathbf{N}_{s,p}(Z) \coloneq \norm{f}_{\mathcal{F}\mathscr{W}^{s,p}}
\\& \mathbf{W}_{s,p}(Z) \coloneq \norm{f}_{\mathscr{W}^{s,p}}.
\end{align*}
We easily verify that
\begin{equation*}
  \mathbf{N}_{s,p}(Z) = \paren*{ \int \abs*{ \paren*{1+\xi^{2}}^{s/2} \Esp*{ \mathrm{e}^{\mathrm{i} \xi Z} } }^{p} \mathrm{d} \xi }^{1/p},
\end{equation*}
where, when $p =\infty$, the above integral is understood as an essential supremum.
Owing to the duality between $\mathscr{L}^{p}$ and $\mathscr{L}^{p'}$, we find that, for $s \in \mathbb{N}$
\begin{equation*}
  \mathbf{W}_{s,p}(Z) = \Sup*{ \abs*{ \Esp*{ \partial^{s} \varphi(Z)} } : \varphi \in \mathscr{C}_{c}^{\infty},\, \norm{\varphi}_{\mathscr{L}^{p'}} \leq 1 }.
\end{equation*}

\subsection{The Wiener space}
All random variables are defined on a sufficiently large probability space $(\Omega, \mathfrak{W}, \prob)$.
Let us first recall some basic notions regarding the Wiener space and the associated Dirichlet form.
We follow \cite{BouleauHirsch} (see also \cite{Nualart,NourdinPeccati}).
The {Wiener space} is the probability space given by the countable product \begin{equation*}
  \bigotimes_{k \in \mathbb{N}} (\mathbb{R}, \mathfrak{B}(\mathbb{R}), \gamma),
\end{equation*}
where $\mathfrak{B}(\mathbb{R})$ is the Borel $\sigma$-algebra of $\mathbb{R}$ and $\gamma$ is the standard Gaussian measure.
We consider the projection maps
\begin{equation*}
  Y_{k} \colon \mathbb{R}^{\mathbb{N}} \ni (x_{i}) \to x_{k}.
\end{equation*}
By construction, under $\gamma^{\mathbb{N}}$, $(Y_{k})$ is a sequence of independent standard Gaussian variables.
Whenever $(X_{i})$ is a sequence of random variables each defined on an independent Wiener space, we can look at the sequence $(X_{i})$ defined on the Wiener space.
Indeed, for each $i \in \mathbb{N}$ writing $(Y_{k,i})_{k}$ for the coordinate system associated with the Wiener space on which $X_{i}$ is defined, we see that $(X_{i})$ is defined on
\begin{equation*}
  \bigotimes_{k,i \in \mathbb{N}} (\mathbb{R}, \mathfrak{B}(\mathbb{R}), \gamma) \simeq \bigotimes_{k \in \mathbb{N}} (\mathbb{R}, \mathfrak{B}(\mathbb{R}), \gamma).
\end{equation*}
We write $\mathbb{L}^{p}$ for the space of $L^{p}$ random variables, measurable with respect to the underlying Wiener space.
We stress that our probability space $\Omega$ is larger than the underlying Wiener space, and contains random variables independent of this underlying Wiener space,
In particular $L^{p}(\prob)$ is larger than $\mathbb{L}^{p}$.

\subsubsection{Dirichlet form on the Wiener space}\label{s:dirichlet-wiener}

\paragraph{Malliavin derivatives}
For \emph{cylindrical random variables}, that are of the form
\begin{equation*}
  X = f(Y_{1}, \dots, Y_{l}), \qquad f \in \mathscr{C}^{\infty}(\mathbb{R}^{l}),\, l \in \mathbb{N}^{*},
\end{equation*}
we define the first and second \emph{Malliavin derivatives}
\begin{align*}
  & \mathsf{D} X \coloneq (\partial_{y_{k}}X)_{k},
\\& \mathsf{D}^{2}X \coloneq (\partial_{y_{k}} \partial_{y_{l}}X)_{kl}.
\end{align*}
These two differential operators can be extended to non-cylindrical random variables through a standard closure argument, since $\mathsf{D}$ and $\mathsf{D}^{2}$ are closable in $\mathbb{L}^{2}(\ell^{2}(\mathbb{N}))$ and $\mathbb{L}^{2}(\ell^{2}(\mathbb{N} \otimes \mathbb{N}))$ respectively, see \cite[Prop.~1.2.1]{Nualart}.
Our construction of Malliavin derivatives coincides with the setting of an \emph{isonormal Gaussian process} over the Hilbert space $\ell^{2}(\mathbb{N})$.
In \cref{s:bouleau-derivative}, choosing the Hilbert space as a Gaussian space leads to the construction of the \emph{Bouleau derivative}, that plays a crucial role in our analysis.
\paragraph{Carré du champ}
Provided a random variable $F$ is smooth enough, for instance we can take $F \in \mathbb{R}[Y_{1}, \dots, Y_{N}]$ for some $N \in \mathbb{N}$, we define its \emph{carré du champ}
\begin{equation*}
  \Gamma(F,F) \coloneq \sum_{k} (\partial_{y_{k}} F)^{2}.
\end{equation*}
Classically, by density and polarization, this quadratic operator can be extended, by a standard closure procedure, to an unbounded quadratic form
\begin{align*}
  & \Gamma \colon \mathbb{L}^{2} \times \mathbb{L}^{2} \to \mathbb{L}^{1},
\\& \dom \Gamma \coloneq \set*{ F \in \mathbb{L}^{2} : \Gamma(F,F) < \infty }.
\end{align*}

\paragraph{Wiener--Dirichlet form}
The carré du champ allows us to define a so-called \emph{Dirichlet form}, that is a closed symmetric non-negative bilinear form defined on the dense subspace $\dom \Gamma \subset \mathbb{L}^{2}$
\begin{equation*}
  \mathcal{E}(F,F) \coloneq \esp \Gamma(F,F).
\end{equation*}
To this Dirichlet form corresponds a unbounded self-adjoint \emph{Markov generator}, sometimes called the \emph{generator of the Ornstein--Uhlenbeck semi-group}
\begin{equation*}
  \mathsf{L} \colon \mathbb{L}^{2} \to \mathbb{L}^{2},
\end{equation*}
with domain $\dom \mathsf{L} \subset \dom \mathcal{E}$, and characterized by the following integration by parts formula
\begin{equation*}
  \mathcal{E}(F, R) = - \Esp*{ F \mathsf{L} R}, \qquad F \in \dom \mathcal{E},\, R \in \dom \mathsf{L}.
\end{equation*}

\paragraph{Smooth random variables}
Carrying out our computations requires to work with elements of a sufficiently large set of random variables on which we can perform certain operations.
Let us thus introduce the set of \emph{smooth random variables} 
\begin{equation*}
  \mathbb{D}^{\infty} \coloneq \set*{ \mathsf{L}^{k} F \in \bigcap_{p \in (1,\infty)} \mathbb{L}^{p} \cap \dom \mathsf{L}, \qquad k \in \mathbb{N} }.
\end{equation*}
Recall that in the definition above $\mathsf{L}^{0}F = F$.
It is known \cite[Chap.\ 2 \S 7]{BouleauHirsch}, that $\mathbb{D}^{\infty}$ is an algebra, and that, whenever $F \in \mathbb{D}^{\infty}$, then, also $\Gamma(F,F) \in \mathbb{D}^{\infty}$.

\paragraph{Important formulas related to the locality}
The Wiener Dirichlet form is \emph{diffusive}, meaning that the carré du champ satisfies a chain rule
\begin{equation}\label{eq:chain-rule-carre-du-champ}
  \Gamma(\varphi(F_{1}, \dots, F_{d}), R) = \sum_{k=1}^{d} \partial_{k} \varphi(F_{1}, \dots, F_{d}) \Gamma(F_{k},R), \qquad F_{1}, \dots, F_{d},\, R \in \dom \mathcal{E},\, \varphi \in \mathscr{C}^{1}_{b};
\end{equation}
and a Leibniz rule
\begin{equation}\label{eq:leibniz-carre-du-champ}
  \Gamma(FR, S) = R\Gamma(F,S) + F\Gamma(R,S), \qquad F,\,R,\,S \in \mathbb{D}^{\infty}.
\end{equation}
In particular, we have the integration by parts formula for $\Gamma$
\begin{equation}\label{eq:leibniz-formula-3}
  \Esp*{ \Gamma(F, S) R } = - \Esp*{ FR \mathsf{L} S } - \Esp*{ \Gamma(R,S) F}, \qquad F,\, R,\, S \in \mathbb{D}^{\infty}.
\end{equation}

\subsubsection{Hypercontractivity on smooth polynomial}\label{s:hypercontractivity}

We finish these reminders by stating some norm equivalence on \emph{smooth polynomials}, those are multi-linear polynomials evaluated in independent copies of elements of $\mathbb{D}^{\infty}$.
Whenever the elements are elements of the \emph{Wiener chaoses} (see \cite{NourdinPeccati} for definitions), this equivalence of norms is well-known.
Because we are working with generic elements of $\mathbb{D}^{\infty}$, we give a complete proof based on hypercontractivity estimates from \cite{MDOInvariance}.

Recall that our Wiener is generated by the independent array $(Y_{i,k})_{ik}$.
We call \emph{smooth basis} a sequence $\mathcal{X}=\{\mathcal{X}_1,\mathcal{X}_2,\ldots\}$ where there exist a finite number of functions $f_{1}, \dots, f_{p}$ such that $\mathcal{X}_i = \set*{f_{1}((Y_{i,k})_{k}), \dots, f_{p}((Y_{i,k})_{k})} \subset \mathbb{D}^{\infty}$.
If $I$ is  finite subset of $\mathbb{N}$, we denote
\begin{equation*}
  \mathcal{X}_I \coloneq \set*{\prod_{i\in I} X_i,\, X_i\in \mathcal{X}_i },
\end{equation*}
and, we define, for all $m\in\mathbb{N}$, $\mathcal{P}_m(\mathcal{X})$ as the vector space spanned by the random variables $X_I\in \mathcal{X}_I$ for $|I|\leq m$.

\begin{lemma}\label{th:polynomial:equivalence-lp}
  If $\mathcal{X}$ is a smooth basis, all the $L^{p}$-norms are equivalent on $\mathcal{P}_m(\mathcal{X})$.
\end{lemma}
\begin{proof}
  It is a consequence of \cite[Props.\ 3.16, 3.11 \& 3.12]{MDOInvariance}.
  The Wiener space or the Dirichlet form are actually not used here.
\end{proof}

\begin{lemma}\label{th:polynomial:dirichlet-continuous}
  Let $\mathcal{X}$ be a smooth basis and $m \in \mathbb{N}$.
  Then, the bilinear form $\mathcal{E}$ is continuous on $(\mathcal{P}_{m}(\mathcal{X}), \norm{\cdot}_{L^{2}})$.
\end{lemma}
\begin{proof}
Let $I$ and $J \subset \mathbb{N}$ with $|I|\leq m$ and $|J| \leq m$.
Take $F_{I}$ and $F_{J}$ respectively in the linear span of $\mathcal{X}_{I}$ and $\mathcal{X}_{J}$.
Namely, there exist finite sets $L$ and $K \subset \mathbb{N}$, $(a_{k}) \in \mathbb{R}^{K}$, $(b_{l}) \in \mathbb{R}^{L}$, and random variables $X_{i}^{k} \in \mathcal{X}_{i}$ ($i \in I, k \in K$) and $X_{j}^{l} \in \mathcal{X}_{j}$ ($j \in J,l \in L$) such that
\begin{equation*}
  F_{I} = \sum_{k \in K} a_{k} X_{I}^{k},
  \qquad \text{and} \qquad
  F_{J} = \sum_{l \in L} b_{l} X_{J}^{l}.
\end{equation*}
From the independence and the centering, we see immediately that
\begin{equation*}
  \Esp*{F_{I} F_{J}} = 1_{I = J} \sum_{k \in K,l \in L} a_{k} b_{l} \prod_{i \in I,j \in J} \Esp*{X_{i}^{k} X_{i}^{l}}.
\end{equation*}
Moreover, by bilinearity and the Leibniz rule for $\Gamma$, we have that
\begin{equation*}
  \Gamma(F_{I}, F_{J}) = \sum_{k \in K,l \in L} \sum_{i \in I,j \in J} X^{k}_{I \setminus \{i\}} X^{l}_{J \setminus \{j\}} \Gamma(X^{k}_{i}, X^{l}_{j}).
\end{equation*}
Since for $i \ne j$, $X_{i} \in \mathcal{X}_{i}$ and $X_{j} \in \mathcal{X}_{j}$ leaves in two independent copies of the Wiener space, we have in this case $\Gamma(X_{i}, X_{j}) = 0$.
Thus taking expectation in the expression above, and using again the independence and the centering, yields
\begin{equation*}
  \Esp*{\Gamma(F_{I}, F_{J})} = 1_{I = J}\sum_{k \in K, l \in L} a_{k} b_{l} \sum_{i \in I} \Gamma(X_{i}^{k}, X_{i}^{l}) \prod_{i' \in I \setminus \{i\}} \Esp*{X_{i'}^{k} X_{i'}^{l}}.
\end{equation*}
Let
\begin{equation*}
  c \coloneq m \sup_{i \in I} \sup_{X_{i}, X_{i}' \in \mathcal{X}_{i}} \frac{\esp \Gamma(X_{i}, X_{i}')}{\esp X_{i} X_{i}'} = m \sup_{X,X' \in \mathcal{X}_{1}} \frac{\esp \Gamma(X_{1}, X_{1}')}{\esp X_{1} X_{1}'}.
\end{equation*}
The equality holds since all the $\mathcal{X}_{i}$'s are copies of each other.
Moreover, $c$ is finite since $\mathcal{X}_{1}$ is a finite set of $\mathbb{D}^{\infty}$.
We have shown, on the one hand that the family $(F_{I})$ is orthogonal for the $\mathbb{L}^{2}$-scalar product as well as with respect to the bilinear form $\mathcal{E}$; and, on the other hand, that
\begin{equation*}
  \mathcal{E}(F_{I}, F_{I}) \leq c \Esp{F_{I}^{2}}.
\end{equation*}
Now, take $F = \sum_{|I| \leq m} F_{I} \in \mathcal{P}_{m}(\mathcal{X})$.
By orthogonality and the above inequality
\begin{equation*}
  \mathcal{E}(F,F) = \sum_I \mathcal{E}(F_I,F_I) \leq C \Esp{F_{I}^{2}} = C \Esp{F^{2}}.
\end{equation*}
This completes the proof.
\end{proof}

\begin{lemma}\label{th:polynomial:l-continuous}
  Let $\mathcal{X}$ be a smooth basis and $m \in \mathbb{N}$.
  Then, there exists a smooth basis $\mathcal{X}'$ such that $\mathsf{L} \colon \mathcal{P}_m(\mathcal{X})\to  \mathcal{P}_m(\mathcal{X}')$.
  Moreover, $\mathsf{L}$ is continuous from $(\mathcal{P}_m(\mathcal{X}),\norm{\cdot}_{L^p})$ to $(\mathcal{P}_m(\mathcal{X}'),\norm{\cdot}_{L^q})$ for any $1\leq p,q< +\infty$.
\end{lemma}
\begin{proof}
  Take $I \subset \mathbb{N}$ with $|I| \leq m$ and $X_{I} \in \mathcal{X}_{I}$.
  For $X$ and $Y \in \mathbb{D}^{\infty}$, we have $\mathsf{L}(XY) = X \mathsf{L} Y + Y \mathsf{L} X + 2 \Gamma(X,Y)$.
  Since whenever $X$ and $Y$ are defined on independent Wiener space, $\Gamma(X,Y) = 0$, we find that
  \begin{equation*}
    \mathsf{L} X_{I} = \sum_{i \in I} X_{I \setminus \{i'\}} \mathsf{L} X_{i}.
  \end{equation*}
The first part of the statement follows by taking $\mathcal{X}' \coloneq \{\mathcal{X}_1',\mathcal{X}_2',\ldots\}$ with
\begin{equation*}
  \mathcal{X}_i'\coloneq \mathcal{X}_i\cup\{\mathsf{L}(X) : X\in \mathcal{X}_i\}, \qquad i \in \mathbb{N}.
\end{equation*}
Next, by \cref{th:polynomial:dirichlet-continuous}, there exists $C > 0$ such that for $F$ in $\mathcal{P}_m(\mathcal{X})$ we have $|\Esp*{F \mathsf{L}(F)}|=\mathcal{E}(F,F)\leq C \norm{F}_{L^2}^2$.
It implies that $\mathsf{L} \colon (\mathcal{P}_m(\mathcal{X}),\norm{\cdot}_{L^2}) \to (\mathcal{P}_m(\mathcal{X}'),\norm{\cdot}_{L^2})$ is continuous.
Since, by \cref{th:polynomial:equivalence-lp}, all the $L^p$ norms are equivalent on $\mathcal{P}_m(\mathcal{X})$ and $\mathcal{P}_m(\mathcal{X}')$, we conclude.
\end{proof}

\begin{lemma}\label{th:polynomial:gamma-continuous}
  Let $\mathcal{X}$ be a smooth basis and $m \in \mathbb{N}$.
  Then, there exists a smooth basis $\mathcal{X}'$ such that $\Gamma \colon \mathcal{P}_m(\mathcal{X})\times  \mathcal{P}_m(\mathcal{X}) \to  \mathcal{P}_{2m-1}(\mathcal{X}')$.
  Moreover,  for any $1\leq p,q< +\infty$, the bilinear operator $\Gamma$ is continuous when seen as an operator
  \begin{equation*}
    (\mathcal{P}_m(\mathcal{X}),\norm{\cdot}_{L^p})\times (\mathcal{P}_m(\mathcal{X}),\norm{\cdot}_{L^p}) \longrightarrow (\mathcal{P}_{2m-1}(\mathcal{X}'),\norm{\cdot}_{L^q}).
  \end{equation*}
\end{lemma}
\begin{proof}
  Using the Leibniz rule for $\Gamma$ and the fact that $\Gamma(X,Y) = 0$ whenever $X$ and $Y$ are defined on different Wiener spaces, we find that
  \begin{equation*}
    \Gamma(X_{I}, Y_{J}) = \sum_{i \in I \cap J} X_{I \setminus \{i\}} Y_{J \setminus \{i\}} \Gamma(X_{i}, Y_{i}), \qquad X_{I} \in \mathcal{X}_{I},\, Y_{J} \in \mathcal{X}_{J}.
  \end{equation*}
This shows that the first part of the statement holds setting $\mathcal{X}' \coloneq \{\mathcal{X}_1',\mathcal{X}_2',\ldots\}$ with 
\begin{equation*}
  \mathcal{X}_i' \coloneq \mathcal{X}_i\cup\{XY-\Esp*{XY}, X,Y\in \mathcal{X}_i\}\cup\{\Gamma(X,Y)-\mathcal{E}(X,Y), X,Y\in \mathcal{X}_i\}.
\end{equation*}
By \cref{th:polynomial:dirichlet-continuous}, there exists $C > 0$ such that for $F \in \mathcal{P}_m(\mathcal{X})$, we have $\norm{\Gamma(F,F)}_{L^{1}} = \mathcal{E}(F,F)\leq C \norm{F}_{L^{2}}^2$.
It implies that $\Gamma \colon (\mathcal{P}_m(\mathcal{X}),\norm{\cdot}_{L^2})\times (\mathcal{P}_m(\mathcal{X}),\norm{\cdot}_{L^2}) \to (\mathcal{P}_m(\mathcal{X}'),\norm{\cdot}_{L^1})$ is continuous.
By \cref{th:polynomial:equivalence-lp}, all the $L^p$ norms are equivalent on $\mathcal{P}_m(\mathcal{X})$ and $\mathcal{P}_{2m-1}(\mathcal{X}')$, and we conclude.
\end{proof}
If $k$ is an integer, we use the notation $\Gamma^k(F)$ for the variable defined by the induction formula $\Gamma^{1} = \Gamma$, and $\Gamma^{k+1}(F) \coloneq \Gamma\left[\Gamma^{k}(F),\Gamma^{k}(F)\right]$, $k \in \mathbb{N}^{*}$.
\begin{lemma}\label{th:polynomial:equivalence}
  Let $\mathcal{X}$ be a smooth basis, $m,k \in \mathbb{N}$, and $1<p<+\infty$.
  Then, there exists a constant $C$ such that for any $F\in \mathcal{P}_m(\mathcal{X})$ with $\Esp*{F^{2}} = 1$, we have $\norm{\Gamma^k(F)}_{L^{p}} \leq C$ and $\norm{\Gamma^k(LF)}_{L^{p}} \leq C$.
\end{lemma}
\begin{proof}
  Follows by successive applications of \cref{th:polynomial:l-continuous,th:polynomial:gamma-continuous}.
\end{proof}

\section{Regularity of laws of a random variables on the Wiener space}

\subsection{Regularity estimates for smooth random variables}\label{s:regularity-independent-sequence}
\subsubsection{Regularity of laws from positivity of the carré du champ}\label{s:regularity-positivity-gamma}

The analysis of Dirichlet forms is relevant for the study of regularity of laws.
Indeed, as it is part of the folklore: if $F \in \mathbb{D}^{\infty}$ and $\Gamma(F,F)^{-1} \in \cap_{1<p<\infty} \mathbb{L}^{p}$, then $F$ has a law that is $\mathscr{C}^{\infty}$.
See \cite[Thm.\ 1.14]{Watanabe} for this statement.
Here, we give a quantitative version of this estimate.

We fix $F \in \mathbb{D}^{\infty}$.
Let us define recursively $\Psi_{0}(F) \coloneq 1$, and
\begin{equation*}
  \Psi_{k+1}(F) \coloneq  \Gamma(F,F) \bracket*{ - \Psi_{k}(F) \mathsf{L} F - \Gamma(F, \Psi_{k}(F)) } + \Psi_{k}(F) \Gamma(F, \Gamma(F,F)), \qquad k \in \mathbb{N}.
\end{equation*}
Since for $F \in \mathbb{D}^{\infty}$, $\mathsf{L} F \in \mathbb{D}^{\infty}$, $\Gamma(F,F) \in \mathbb{D}^{\infty}$, and that $\mathbb{D}^{\infty}$ is an algebra, an induction yields that $\Psi_{k}(F) \in \mathbb{D}^{\infty}$ for all $k \in \mathbb{N}$.
For convenience, write
\begin{equation*}
  \boldsymbol{\Psi}_{k}(F) \coloneq \norm{\Psi_{k}(F)}_{\mathbb{L}^{2}}.
\end{equation*}
The quantities $\Psi_{k}$'s allows us to perform integration by parts and thus to control the regularity of the density of $F$.
\begin{proposition}\label{th:integration-by-parts-gamma}
  Let $F \in \mathbb{D}^{\infty}$ smooth such that $\Gamma(F,F)^{-1} \in \mathbb{L}^{4k}$ for some $k \in \mathbb{N}$.
  Then,
\begin{equation}\label{eq:ipp-phi-k-gamma}
  \Esp*{ \varphi^{(k)}(F) } = \Esp*{ \varphi(F) \frac{\Psi_{k}(F)}{\Gamma(F,F)^{2k}} }, \qquad \varphi \in \mathscr{C}^{\infty}_{c}.
\end{equation}
  In particular, we have that
  \begin{equation}\label{eq:bound-sobolev-gamma}
    \mathbf{W}_{k,1}(F) \lesssim \boldsymbol{\Psi}_{k}(F) \Esp*{ {\Gamma(F,F)^{-4k}} }^{1/2}.
  \end{equation}
\end{proposition}
\begin{proof}
  Since $\varphi$ is bounded, by the Cauchy--Schwarz inequality the right-hand sides in \cref{eq:ipp-phi-k-gamma,eq:bound-sobolev-gamma} are well-defined.
  For short we will write in this proof $\Psi_{k} \coloneq \Psi_{k}(F)$, $S \coloneq \Gamma(F,F)$, and
  \begin{equation*}
    W_{k} \coloneq \frac{\Psi_{k}}{\Gamma(F,F)^{2k}} = \frac{\Psi_{k}}{S^{2k}}.
  \end{equation*}
  We establish \cref{eq:ipp-phi-k-gamma} by induction.
  For $k = 0$, it is trivial.
  Assume that $S^{-1} \in \mathbb{L}^{4k+4}$ and that the claim is established up to some $k \in \mathbb{N}$.
  By \cref{eq:ipp-phi-k-gamma}, we get that
  \begin{equation}\label{eq:ipp-gamma-induction-step}
    \Esp*{ \varphi^{(k+1)}(F) } = \Esp*{ \varphi'(F) W_{k} }.
  \end{equation}
  Using the diffusion property \cref{eq:chain-rule-carre-du-champ}, we get that
  \begin{equation}\label{eq:ipp-gamma-appears}
    \Esp*{ \varphi'(F) W_{k} } = \Esp*{ \Gamma(\varphi(F), F) \frac{\Psi_{k}}{S^{2k+1}} }.
  \end{equation}
  We shall now carry out computations as if $S \in \mathbb{D}^{\infty}$.
  On the one hand, we have that $S^{-2k-1} \in \mathbb{L}^{2}$.
  On the other hand, by the chain rule \cref{eq:chain-rule-carre-du-champ}, we find that
  \begin{equation*}
    \Gamma(S^{-2k-1}, S^{-2k-1}) = (2k+1)^{2} S^{-4k-4} \Gamma(S, S).
  \end{equation*}
  Let us mention that this \emph{a priori} formal computation can be made rigorous by doing the computation at the level of the smooth random variable $\paren*{ S + \varepsilon}^{-2k-1}$, and then letting $\varepsilon \to 0$.
  Since $\Gamma(S, S) \in \mathbb{D}^{\infty}$, this shows that
  \begin{equation*}
    \Gamma(S^{-2k-1}, S^{-2k-1}) \in \mathbb{L}^{1}.
  \end{equation*}
  Thus, $S^{-2k-1} \in \mathbb{D}^{1,2} = \dom \mathcal{E}$.
  Since $\Psi_{k} \in \mathbb{D}^{\infty}$, we have shown that $\Psi_{k}S^{-2k-1} \in \dom \mathcal{E}$.
  Since we also have that $\varphi(F)$ and $F \in \mathbb{D}^{\infty}$, all the following computations make sense.
  By integration by parts and the Leibniz rule, we get that
  \begin{equation*}
    \begin{split}
      - \Esp*{ \varphi(F) \frac{\Psi_{k}}{S^{2k+1}} \mathsf{L} F } &= \Esp*{ \Gamma\paren*{\varphi(F) \frac{\Psi_{k}}{S^{2k+1}}, F} }
                                                              \\&= \Esp*{ \Gamma(\varphi(F), F) \frac{\Psi_{k}}{S^{2k+1}} } + \Esp*{ \varphi(F) \Gamma\paren*{\frac{\Psi_{k}}{S^{2k+1}}, F} }.
    \end{split}
  \end{equation*}
  Combining with \cref{eq:ipp-gamma-induction-step,eq:ipp-gamma-appears} yields
  \begin{equation*}
    \Esp*{ \varphi^{(k+1)}(F) } = \Esp*{ \varphi(F) \bracket*{ - \Psi_{k}S^{-2k-1} \mathsf{L}F - \Gamma(\Psi_{k}S^{-2k-1}, F) } }.
  \end{equation*}
  By the Leibniz rule and the chain rule, we find that
  \begin{equation*}
    \Gamma(\Psi_{k}S^{-2k-1},F) = S^{2k-1} \Gamma(\Psi_{k}, F) - (2k+1) \Psi_{k} \Gamma(S,F) S^{-2k-2}.
  \end{equation*}
  Recalling that $S = \Gamma(F,F)$ this concludes the proof in view of the definition of the $\Psi_{k}$'s.
\end{proof}

\subsubsection{The Bouleau derivatives}\label{s:bouleau-derivative}
We now use the formalism of Dirichlet forms on the Wiener space to derive regularity estimates for the density of a smooth random variable.
Roughly speaking, we construct conditionally \emph{Gaussian random variables} encoding the properties of the carré du champ.
As anticipated, the following objects play a pivotal role in our analysis.
\begin{definition}
  The first and second \emph{Bouleau derivatives} of $X \in \mathbb{D}^{\infty}$ are respectively defined by
\begin{align}
  &\label{def:diese-1} \sharp_{G} X \coloneq \mathsf{D} X \cdot \vec{G} = \sum_{k} (\partial_{y_{k}}X) G_{k},
\\&\label{def:diese-2} \sharp_{H} \sharp_{G} X \coloneq \vec{H} \cdot \mathsf{D}^{2}X \vec{G} = \sum_{kl} (\partial_{y_{k}} \partial_{y_{l}} X) G_{k} H_{l},
\end{align}
where $\vec{G} = (G_{k})$ and $\vec{H} = (H_{l})$ are sequences of independent standard Gaussian variables, such that the array $(G_{k}, H_{l})$ is independent of the underlying Wiener space.
\end{definition}
By construction, $\sharp_{G} X$ has the same law as $\Gamma(X,X)^{1/2}N$ where $N$ is a standard Gaussian variable independent of $\Gamma(X,X)$.
\begin{remark}
The Bouleau derivatives are nothing but the celebrated \emph{Malliavin derivatives} when we take derivative in the direction of an independent Gaussian space.
\end{remark}
It is convenient for us to identify $\sharp_{G}X$ and $\sharp_{H}\sharp_{G}X$ with element of the Wiener space generated by the underlying Wiener space, $G$ and $H$.
With this identification, we have $\sharp_{G}X$ and $\sharp_{H} \sharp_{G} X \in \mathbb{D}^{\infty}$.

\subsubsection{Regularity of the random variable through the Bouleau derivative}\label{s:regularity-via-bouleau-derivative}

Our general principle states that the regularity of the law of $\sharp_{G}F$ controls that of the law $F$.
\begin{theorem}\label{th:sobolev-regularity-derivative}
  Let $F \in \mathbb{D}^{\infty}$.
  Then,
  \begin{equation}\label{eq:sobolev-regularity-derivative}
    \begin{split}
      \mathbf{W}_{q,1}(F) &\lesssim \boldsymbol{\Psi}_{q}(F) \mathbf{N}_{8q+1,\infty}(\sharp_{G}F) ^{1/2}
                        \\& \lesssim \boldsymbol{\Psi}_{q}(F) \mathbf{W}_{8q+1,1}(\sharp_{G}F) ^{1/2}.
    \end{split}
  \end{equation}
\end{theorem}

\begin{proof}
Let $\xi > 0$ and $\lambda > 0$.
By Markov's inequality, we have that:
\begin{equation*}
  \Prob*{ \Gamma(F,F) < \frac{1}{\xi} } = \Prob*{ \Exp*{ - \frac{\lambda^{2}}{2} \Gamma(F,F) } > \mathrm{e}^{-\lambda^{2}/(2\xi) } } \leq \mathrm{e}^{\lambda^{2}/(2 \xi)} \Esp*{ \Exp*{- \frac{\lambda^{2}}{2} \Gamma(F,F) } }.
\end{equation*}
By construction, $\sharp_{G} F$ is a conditionally Gaussian variable with conditional variance $\Gamma(F,F)$, and we have that
\begin{equation}\label{eq:laplace-fourier}
  \Esp*{\Exp*{ - \frac{\lambda^{2}}{2} \Gamma(F,F)} } = \Esp*{ \Exp*{ \mathrm{i} \lambda \sharp_{G} F} }.
\end{equation}
Thus, taking $\lambda = \sqrt{2\xi}$ we have that
\begin{equation*}
  \Prob*{ \Gamma(F,F) < \frac{1}{\xi} } \leq \mathrm{e} \Esp*{ \Exp*{ \mathrm{i} (2\xi)^{1/2} \sharp_{G} F } }.
\end{equation*}
It follows that
\begin{equation}\label{eq:negative-moment-gamma-fourier}
  \begin{split}
    \Esp*{ \Gamma(F,F)^{-q} } &= q \int_{0}^{\infty} \xi^{q-1} \Prob*{ \Gamma(F,F) < \frac{1}{\xi} } \mathrm{d}\xi 
                            \\&\leq \frac{\mathrm{e}q}{2}  \int_{0}^{\infty} \xi^{2q-1} \Esp*{ \Exp*{ \mathrm{i} \xi \sharp_{G}F}  } \mathrm{d} \xi 
                            \\&\leq \frac{\mathrm{e}q}{2} \paren*{\int_{0}^{\infty} \frac{\xi^{2q-1}}{(1+\xi^{2})^{q+1/2}} \mathrm{d} \xi} \mathbf{N}_{2q +1, \infty}(\sharp_{G}F) .
  \end{split}
\end{equation}
By \cref{th:integration-by-parts-gamma},
\begin{equation*}
\mathbf{W}_{q,1}(F) \lesssim \boldsymbol{\Psi}_{q}(F) \Esp*{ \Gamma(F,F)^{-4q} }^{1/2}.
\end{equation*}
From which, together with \cref{eq:negative-moment-gamma-fourier}, we conclude the first inequality in \cref{eq:sobolev-regularity-derivative}.
The second inequality follows by \cref{eq:embedding-fourier-sobolev}.
\end{proof}

We can iterate \cref{th:sobolev-regularity-derivative}.
\begin{theorem}\label{th:sobolev-regularity-iterated-gradient}
  Let $F \in \mathbb{D}^{\infty}$.
  We have that
  \begin{equation*}
    \mathbf{W}_{q,1}(F) \lesssim \boldsymbol{\Psi}_{q}(F) \boldsymbol{\Psi}_{8q+1}(\sharp_{G}F) \mathbf{N}_{64q+9, \infty}(\sharp_{H}\sharp_{G}F)^{1/4}.
  \end{equation*}
\end{theorem}

\begin{remark}
  The recursive definition of the $\Psi_{q}$ involves repeated applications of the operators $\mathsf{L}$ and $\Gamma$.
  As mentioned above, we look at $\sharp_{G}F$ as an element of the  Wiener space, and we have $\sharp_{G}F \in \mathbb{D}^{\infty}$.
In particular, $\Psi_{q}(\sharp_{G}F)$ makes sense.
We stress however, that the $\mathsf{L}$ and $\Gamma$ are applied to all the variables, including $G$.
\end{remark}

\begin{proof}
  In view of \cref{th:sobolev-regularity-derivative}, it suffices to bound $\mathbf{W}_{r,1}(\sharp_{G}F)$, with $r\coloneq 8q+1$.
  To do so, we apply \cref{th:integration-by-parts-gamma} to $\sharp_{G}F$, which is indeed an element of $\mathbb{D}^{\infty}$ when extending the underlying Wiener space with $G$.
  We have that
  \begin{equation*}
    \bar{\Gamma}(\sharp_{G}F, \sharp_{G}F) = \Gamma(\sharp_{G}F, \sharp_{G}F) + \Gamma(F,F),
  \end{equation*}
  where $\bar{\Gamma}$ means we compute the carré du champ with respect to all the variables, while in $\Gamma$ we only compute the carré du champ with respect to the initial underlying Wiener space.
  Thus, rather than controlling the negative moments of $\bar{\Gamma}(\sharp_{G}F, \sharp_{G}F)$ it is sufficient to control that of $\Gamma(\sharp_{G}F, \sharp_{G}F)$.
  Since we have again the Fourier--Laplace identity
  \begin{equation*}
    \Esp*{ \Exp*{ - \frac{\lambda^{2}}{2} \Gamma(\sharp_{G}F, \sharp_{G}F) } } = \Esp*{ \Exp*{ \mathrm{i} \lambda \sharp_{H} \sharp_{G} F } },
  \end{equation*}
  repeating the computations leading to \cref{eq:negative-moment-gamma-fourier} yields
  \begin{equation*}
    \Esp*{\Gamma(\sharp_{G}F,\sharp_{G}F)^{-q}} \lesssim \mathbf{N}_{2q+1}(\sharp_{H} \sharp_{G}F).
  \end{equation*}
  We conclude by \cref{th:integration-by-parts-gamma}.
\end{proof}

\subsection{Regularity for smooth variables of a smooth sequence}
\subsubsection{Smooth random variables measurable with respect to a smooth sequence}\label{s:sequence-d-infini}
We want to study smooth random variables of a particular form.
Namely, we fix $X \in \mathbb{D}^{\infty}$ and we consider $(X_{i})$ a sequence of independent copies of $X$.
In the sense that, as in \cref{s:hypercontractivity}, there exists a function $x$ such that $X_{i} = x((Y_{i,k})_{k})$, where we recall that our Wiener space is generated by the independent array $(Y_{i,k})_{ik}$.

We consider a random variable $F = F(X_{1}, X_{2}, \dots)$.
We moreover assume that $F \in \mathbb{D}^{\infty}$.
We write $\mathbb{D}^{\infty}_{X}$ for the space of such random variables.
Since $F \in \mathbb{D}^{\infty}$, the definitions of the Bouleau derivatives still apply to $F$.
For the reader convenience, we give the explicit expression of $\sharp_{G} F$ and $\sharp_{H} \sharp_{G}F$.
It is now convenient to consider an array of independent Gaussian variables $(G_{i,k}, H_{j,k})_{i,j,k}$, also independent of the $(Y_{i,k})_{i,k}$, and to group them vectorially
\begin{equation*}
  \vec{G}_{i} \coloneq (G_{i,k})_{k}, \qquad \vec{H}_{j} \coloneq (H_{j,k})_{k}.
\end{equation*}
In this case
\begin{align}
  &\label{eq:sharp-cylinder} \sharp_{G} F = \sum_{i \in \mathbb{N}} (\partial_{x_{i}} F) \mathsf{D} X_{i} \cdot \vec{G}_{i} = \sum_{i,k \in \mathbb{N}} (\partial_{x_{i}}F) (\partial_{y_{i,k}}X_{i}) G_{i,k}.
\\& \label{eq:iterated-sharp-cylinder} \sharp_{H} \sharp_{G} F = \sum_{i,j \in \mathbb{N}} (\partial_{x_{i}} \partial_{x_{j}}F) (\mathsf{D} X_{i} \cdot \vec{G}_{i}) (\mathsf{D} X_{j} \cdot \vec{H}_{j}) + \sum_{i \in \mathbb{N}} (\partial_{x_{i}} F) \psh{\vec{H}_{i}}{\mathsf{D}^{2}X_{i} \vec{G}_{i}}.
\end{align}
\subsubsection{Estimating the regularity of the iterated gradient by Gaussian analysis}\label{s:iterated-gradient-square-gaussian}

The random variable $\sharp_{G} \sharp_{H} F$ is conditionally a Gaussian quadratic form, that is a Gaussian quadratic form with random independent coefficients.
More precisely, there exists a random linear operator, measurable with respect to the underlying Wiener space, $\mathsf{A}$ such that
\begin{equation*}
  \sharp_{H} \sharp_{G} F = \psh*{G}{\mathsf{A} H}.
\end{equation*}
The explicit formula for $\mathsf{A}$ is involved but could be inferred from \cref{eq:iterated-sharp-cylinder}; we do not write the formula as we show below that $\mathsf{A}$ has the same law as another random operator, more tractable for our needs.
At the intuitive level, looking at $\sharp_{H} \sharp_{G}F$ as a Gaussian quadratic forms with random independent coefficients allows to control its Fourier transform with spectral remainders of an operator related to $\mathsf{A}$.
In this section, we follow this idea in order to deduce exploitable bounds from \cref{th:sobolev-regularity-iterated-gradient}.
Let us start with an estimation of the regularity of the law of a Gaussian quadratic form, that is with deterministic coefficients.

\begin{lemma}\label{th:regularity-gaussian-quadratic-form}
  Let $N = (N_{i})$ and $W = (W_{i})$ be two independent sequences of independent standard Gaussian variables and $\mathsf{A}$ be a deterministic symmetric Hilbert--Schmidt operator.
  Then,
  \begin{equation*}
    \mathbf{N}_{q/2,\infty}(\psh{N}{\mathsf{A} W}) \lesssim \mathcal{R}_{q}(\mathsf{A})^{-1/4}, \qquad q \in \mathbb{N}.
  \end{equation*}
\end{lemma}
\begin{proof}
  We first prove the claim for the non-independent case $W = N$ and then explain how it is sufficient to conclude.
  Write $F \coloneq \psh{N}{\mathsf{A}N}$.
  By diagonalization we have $\mathsf{A} = \mathsf{P}^{T} \Lambda \mathsf{P}$ where $\Lambda$ is the diagonal operator of real eigenvalues of $\mathsf{A}$ and $\mathsf{P}$ is an orthogonal operator.
  By the invariance of Gaussian measures under isometries, we find that
  \begin{equation*}
    F = \psh{\mathsf{P} N}{\Lambda \mathsf{P}N} \overset{\law}{=} \psh{N}{\Lambda N}.
  \end{equation*}
  Recalling that
  \begin{equation*}
    \Esp*{ \mathrm{e}^{\mathrm{i} \xi N_{1}^{2} } } = \paren*{1-2\mathrm{i}\xi}^{-1/2},
  \end{equation*}
  we get that
  \begin{equation*}
    \abs*{ \Esp*{ \mathrm{e}^{\mathrm{i} \xi F} } } = \prod_{k} \abs*{ 1 -2 \mathrm{i}\xi \lambda_{k} }^{-1/2} = \prod_{k} (1 + 4 \xi^{2} \lambda_{k}^{2})^{-1/4}.
  \end{equation*}
  Developing the product yields
  \begin{equation*}
    \prod_{k} \paren*{ 1 + 4\xi^{2} \lambda_{k}^{2} } = 1 + \sum_{p} 4^{p} \xi^{2p} \mathcal{R}_{p}(\mathsf{A}).
  \end{equation*}
  Finally, we have that
  \begin{equation*}
    \abs*{ \Esp*{ \mathrm{e}^{\mathrm{i} \xi F} } } \lesssim (1+\xi)^{-q/2} \mathcal{R}_{q}(\mathsf{A})^{-1/4},
  \end{equation*}
  from which the claim follows in the case $N = W$.
  Now, consider two independent sequences $N$ and $W$.
  Write
  \begin{equation*}
    \mathsf{A}^{\oslash 2} \coloneq
    \paren*{%
      \begin{array}{c|c}
        0 & \mathsf{A} \\
        \hline
        \mathsf{A} & 0
      \end{array}
    }
  \end{equation*}
The characteristic polynomial of $\mathsf{A}^{\oslash 2}$ is given by
\begin{equation*}
\det \paren*{%
  \begin{array}{c|c}
    -\lambda \mathsf{Id} & \mathsf{A} \\
    \hline
    \mathsf{A} & -\lambda \mathsf{Id}
  \end{array}
} = \det(\lambda^{2} \mathsf{Id} - \mathsf{A}^{2}) = \det(\lambda \mathsf{Id} - \mathsf{A}) \det(\lambda \mathsf{Id} + \mathsf{A}).
\end{equation*}
Thus the spectrum of $\mathsf{A}^{\oslash 2}$ is obtained by duplicating the spectrum of $\mathsf{A}$ up to a sign, and it follows that $\mathcal{R}_{q}(\mathsf{A}^{\oslash 2}) \geq \mathcal{R}_{q}(\mathsf{A})$.
Since
\begin{equation*}
  \psh*{N}{\mathsf{A} W} = \psh*{\begin{pmatrix}N\\W\end{pmatrix}}{\mathsf{A}^{\oslash 2} \begin{pmatrix}N\\W\end{pmatrix}},
\end{equation*}
we conclude by the previous case.
\end{proof}

We now derive the pivotal abstract result of this paper.
We write, for $i \in \mathbb{N}$, $\Gamma_{i} \coloneq \Gamma(X_{i}, X_{i})$, and, for $I \subset \mathbb{N}$, $\Gamma_{I} \coloneq \prod_{i \in I} \Gamma_{i}$.
Let us consider the random symmetric Hilbert--Schmidt operator given by the Hessian of $F$ with respect to the $X_{i}$'s, namely
\begin{equation*}
  \nabla^{2}F = (\partial_{x_{i}} \partial_{x_{j}} F)_{i,j}.
\end{equation*}
We define the random non-negative coefficients
\begin{equation*}
  b_{IJ} \coloneq (\det (\nabla^{2}F)(I,J))^{2}, \qquad I, J \subset \mathbb{N},
\end{equation*}
where we recall that $(\nabla^{2}F)(I,J)$ stands for the extracted sub-matrix from $\nabla^{2}F$.
The following multi-linear polynomial in the $\Gamma_{i}$'s allows to control the regularity of smooth random variables
\begin{equation}\label{eq:sq}
  \mathcal{S}_{q}(F) \coloneq \sum_{|I| = |J| = q} b_{IJ} \Gamma_{I} \Gamma_{J} 1_{I \cap J = \emptyset}
\end{equation}

\begin{proposition}\label{th:regularity-negative-moments-spectral-remainder}
  Let $F \in \mathbb{D}^{\infty}_{X}$, we have
  \begin{equation}\label{eq:regularity-spectral-remainders-simplified}
    \mathbf{W}_{q,1}(F) \lesssim \boldsymbol{\Psi}_{q}(F) \boldsymbol{\Psi}_{8q+1}(\sharp_{G}F) \Esp*{ \mathcal{S}_{128q+18}(F)^{-1/4}}^{1/4}.
  \end{equation}
\end{proposition}

\begin{proof}
  For all $i \in \mathbb{N}$, consider a random orthogonal transformation $\mathsf{P}_{i} \colon \ell^{2}(\mathbb{N}) \to \ell^{2}(\mathbb{N})$ that depends only on the underlying $(Y_{k,i})_{k}$ but, in particular, not on $G$ and $H$ sending $\vec{e} = (e_{1}, e_{2}, \dots)$ the orthonormal basis of $\ell^{2}(\mathbb{N})$ on $\widetilde{e}_{i} \coloneq \mathsf{P}_{i} \vec{e}$ where we only impose
  \begin{equation*}
    \widetilde{e}_{i,1} = \Gamma_{i}^{-1/2}  (\mathsf{D} X_{i} \cdot \vec{e}).
  \end{equation*}
  We can fulfil this constraint since $\tilde{e}_{i,1}$ has almost surely norm $1$.
  Let us write
  \begin{equation*}
    \widetilde{G}_{i} \coloneq \mathsf{P}_{i} \vec{G}_{i}, \qquad \widetilde{H}_{i} \coloneq \mathsf{P}_{i} \vec{H}_{i}.
  \end{equation*}
  Using \cref{eq:iterated-sharp-cylinder}, and then subsisting this new basis we find 
  \begin{equation*}
    \begin{split}
    \sharp_{H} \sharp_{G} F &= \sum_{i, j} (\partial_{x_{i}} \partial_{x_{j}}F) (\mathsf{D} X_{i} \cdot \vec{G}_{i}) (\mathsf{D} X_{j} \cdot \vec{H}_{j}) + \sum_{i \in \mathbb{N}} (\partial_{x_{i}} F) \vec{H}_{i} \cdot (\mathsf{D}^{2}X_{i} \vec{G}_{i}).
                          \\&= \sum_{i,j} (\partial_{x_{i}} \partial_{x_{j}} F) \Gamma_{i}^{1/2} \widetilde{G}_{i,1} \Gamma_{j}^{1/2} \widetilde{H}_{j,1} + \sum_{i} \mathsf{P}_{i}^{-1} \widetilde{H}_{i} \cdot (\mathsf{D}^{2}X_{i}) \mathsf{P}_{i}^{-1} \widetilde{G}_{i}
    \end{split}
  \end{equation*}
  Up to relabelling of indices, we can find random operators $\mathsf{A}$ and $\mathsf{B}$ such that
  \begin{align*}
    & \psh{\widetilde{H}}{\mathsf{A} \widetilde{G}}  = \sum_{i,j} (\partial_{x_{i}} \partial_{x_{j}} F) \Gamma_{i}^{1/2} \widetilde{G}_{i,1} \Gamma_{j}^{1/2} \widetilde{H}_{j,1}
  \\& \psh{\widetilde{H}}{\mathsf{B} \widetilde{G}} = \sum_{i} \mathsf{P}_{i}^{-1} \widetilde{H}_{i} \cdot (\mathsf{D}^{2}X_{i}) \mathsf{P}_{i}^{-1} \widetilde{G}_{i}
  \end{align*}
  The explicit formulas for $\mathsf{A}$ and $\mathsf{B}$ are not relevant.
  Define the matrix $\mathsf{C} \coloneq \mathsf{A} + \mathsf{B}$.
  Using \cref{th:sobolev-regularity-iterated-gradient,th:regularity-gaussian-quadratic-form} yields
\begin{equation*}
  \mathbf{W}_{q,1}(F) \lesssim \boldsymbol{\Psi}_{q}(F) \boldsymbol{\Psi}_{8q+1}(\sharp_{G}F) \Esp*{ \mathcal{R}_{128q+18}\paren*{\mathsf{C}}^{-1/4} }^{1/4}.
\end{equation*}
Let $q' \coloneq 128q+18$.
  The matrix $\mathsf{B}$ is diagonal, since it only involves terms of the form $\widetilde{G}_{i}$ and $\widetilde{H}_{i}$ with the same index $i$.
  In particular, whenever $I \cap J = \emptyset$, we find $\mathsf{B}(I,J) = 0$.
Thus, by the Cauchy--Binet formula \cref{th:cauchy-binet},
\begin{equation*}
  \mathcal{R}_{q'}(\mathsf{C}) \geq \sum_{|I| = |J| = q'} (\det \mathsf{C}(I,J))^{2} 1_{I \cap J = \emptyset} = \sum_{|I| = |J| = q'} (\det \mathsf{A}(I,J))^{2} 1_{I \cap J = \emptyset}.
\end{equation*}
Consider $\Gamma = \mathrm{diag}(\Gamma_{1}, \Gamma_{2}, \dots)$, up to a suitable choice of indices, we have
\begin{equation*}
  \mathsf{A} = \Gamma^{1/2} (\nabla^{2}F) \Gamma^{1/2}.
\end{equation*}
Using that $\Gamma$ is diagonal, we find
\begin{equation*}
  \det \mathsf{A}(I,J) = \Gamma_{I}^{1/2} (\det (\nabla^{2}F)(I,J)) \Gamma_{J}^{1/2}.
\end{equation*}
This completes the proof.
\end{proof}

\subsection{Regularity of laws of random quadratic forms}\label{s:regularity-quadratic-form}

In this section, we use our general theorem in order to derive sufficient conditions for the regularity of random quadratic forms.
More precisely, as in \cref{s:sequence-d-infini}, we consider let $X \in \mathbb{D}^{\infty}$ and $(X_{i})$ a sequence of independent variables with the same law as $X$.
We also take $\mathsf{A} = (a_{ij})$ a deterministic symmetric Hilbert--Schmidt operator on $\ell^{2}(\mathbb{N})$ with $a_{ii} = 0$ for all $i \in \mathbb{N}$, and we consider the homogeneous quadratic form
\begin{equation}\label{eq:q-quadratic}
  Q_{\mathsf{A}} \coloneq \psh{\mathsf{A}X}{X} = \sum_{i \ne j} a_{ij} X_{i} X_{j}.
\end{equation}

Regularity on the law of $Q_{\mathsf{A}}$ will follow from regularity on the law of $X_{1}$.

\begin{assumption}\label{ass:small-ball-gamma}
There exists $\theta > 0$ such that
\begin{equation}\label{eq:small-ball-theta}
  \Prob*{\Gamma(X_{1}, X_{1}) \leq \varepsilon} \lesssim \varepsilon^{\theta}, \qquad \varepsilon > 0.
\end{equation}
\end{assumption}

\begin{remark}
  \cref{ass:small-ball-gamma} is equivalent to $\Esp*{ \Gamma(X_{1}, X_{1})^{-p_{0}} } < \infty$ for some $p_{0} >0$.
  In view of \cref{s:regularity-positivity-gamma} this implies that $X_{1}$ has a density in some Sobolev space $\mathscr{W}^{s_{0}, 1}$ for some $s_{0} > 0$.
  We stress that we allow $s_{0} < 1$.
\end{remark}

Our main result shows that the law of $Q_{\mathsf{A}}$ is regular provided the influence of $\mathsf{A}$ is small enough.
The rest of this section is devoted to the proof of this theorem.
\begin{theorem}\label{th:regularity-quadratic-form}
  Under \cref{ass:small-ball-gamma}.
  Let $q \in \mathbb{N}$ and $\delta > 0$.
  There exists $\tau_{q,\delta} > 0$ such that for all $\mathsf{A}$ with
  \begin{align}
    & \label{ass:normalized-A} \tr \mathsf{A}^{2} = 1,
  \\& \label{ass:positivity-spectral-remainder-A}\mathcal{R}_{128q+18}(\mathsf{A}) > \delta,
  \\& \label{ass:small-influence-A} \tau(\mathsf{A}) < \tau_{q,\delta},
  \end{align}
the law of $Q_{\mathsf{A}}$ is in $\mathscr{W}^{q,1}$, and in particular, it is in $\mathscr{C}^{\lfloor q-1 \rfloor}$.
Moreover, for some $\eta_{q} > 1/4$,
\begin{equation*}
  \mathbf{W}_{q,1}(Q_{\mathsf{A}}) \lesssim \frac{1}{\delta^{\eta_{q}} }.
\end{equation*}
\end{theorem}

\begin{remark}
  For the sake of simplicity, we did not keep track of the explicit constant $\tau_{q,\delta}$.
 It could however be recovered from a careful examination of our computations.
\end{remark}

\subsubsection{A general small ball estimate for multi-linear polynomials}

As could be anticipated from the previous sections our analysis is concerned with the positivity of some random quadratic forms.
We thus start with some auxiliary results of independent interest.
A \emph{multi-linear polynomial} is a function
\begin{equation*}
p(x) \coloneq \sum_{I \subset \mathbb{N}} a_{I} x_{I}, \qquad x \in \mathbb{R}^{\infty},
\end{equation*}
where $x_{I} \coloneq \prod_{i \in I} x_{i}$, and $(a_{I} : I \in \mathbb{N})$ are real coefficients such that $a_{I} = 0$ for all $I$ with $\abs{I} > d$ for some $d \in \mathbb{N}$.
We call the \emph{degree} of the polynomial the largest $d$ such that $a_{I} \ne 0$ for some $\abs{I} = d$.
\begin{proposition}\label{th:small-ball-l2}
  Let $(U_{i})$ be a sequence of independent and identically distributed random variables with finite third moment, and such that there exists $\theta > 0$ with
  \begin{equation}\label{eq:small-ball-U}
    \sup_{b \in \mathbb{R}} \Prob*{ \abs{U_{1} - b} \leq \varepsilon } \lesssim \paren*{\frac{\varepsilon}{\Var*{U_{1}}^{1/2}}}^{\theta}, \qquad \varepsilon \geq 0.
  \end{equation}
  For all $d \in \mathbb{N}$, let
  \begin{equation}\label{eq:def-theta}
    \theta_{d} \coloneq \paren*{\frac{\theta}{2} \wedge \frac{1}{32}} \frac{1}{4^{d-1}}.
  \end{equation}
  Then, for every $p$ multi-linear polynomial of degree $d$
  \begin{equation}\label{eq:small-ball-P-l2}
    \sup_{b \in \mathbb{R}} \Prob*{ \abs{p(U) - b} \leq \varepsilon } \lesssim \paren*{\frac{\varepsilon}{\Var*{p(U)}^{1/2}}}^{\theta_{d}}, \qquad \varepsilon \geq 0.
  \end{equation}
\end{proposition}

\begin{remark}
  We did not try to optimize the constant $\theta_{d}$.
  Finding the optimal exponent would be an interesting problem, not only in connexion with this work but also regarding generalization of the inequality of \citeauthor{CarberyWright} \cite{CarberyWright}.
\end{remark}

\begin{proof}
  By homogeneity, we assume that $U_{1}$ is centered and with variance $1$, and that $\Var{p(U)} = \sum a_{I}^{2} = 1$.
  We proceed by induction on $d$.
  In the following, we understand a multi-linear polynomial of degree $0$ as a constant.
  We consider the weaker property: for every multi-linear polynomial of degree $d \geq 0$,
  \begin{equation}\label{eq:small-ball-P-l2-weak}
    \Prob*{ \abs{p(U)} \leq \eta } \lesssim \eta^{\theta_{d}}, \qquad \eta \geq 0.
  \end{equation}
  For $d =0$, it is clear that \cref{eq:small-ball-P-l2-weak} holds with $\theta_{0} = 1$ (or any other positive real number) as long as $p \ne 0$.
  We also have that \cref{eq:small-ball-P-l2} implies \cref{eq:small-ball-P-l2-weak}.
  Let us assume that we have proven \cref{eq:small-ball-P-l2-weak} for some $d \geq 0$ and show that \cref{eq:small-ball-P-l2} holds for $d+1$.

  \paragraph{Large influence estimate}
  Fix $i \in \mathbb{N}$.
  Let us write
  \begin{align*}
    & S_{i} \coloneq \sum_{I \ni i} a_{I} U_{I \setminus \{i\}},
  \\& R_{i} \coloneq \sum_{I \not\ni i} a_{I}U_{I}.
  \end{align*}
  In this way, we have that $p(U) = U_{i} S_{i} + R_{i}$, and $S_{i}$ and $R_{i}$ are independent of $U_{i}$.
  When $d =0$, we understand $U_{\emptyset} = 1$, and thus $S_{i} = a_{i}$.
  By \cref{eq:small-ball-U}, it follows that
  \begin{equation*}
    \begin{split}
      \Prob*{ \abs{p(U) - a} \leq \varepsilon } &= \Prob*{ \abs*{ U_{i} S_{i} + R_{i} - a} \leq \varepsilon,\, S_{i} > \eta } + \Prob*{ \abs*{ p(U)- a} \leq \varepsilon,\, S_{i} \leq \eta } 
                             \\&\lesssim \paren*{\frac{\varepsilon}{\eta}}^{\theta} + \Prob*{S_{i} \leq \eta}.
    \end{split}
  \end{equation*}
  Since $S_{i}$ is a multi-linear polynomial of degree $d$, by the induction hypothesis \cref{eq:small-ball-P-l2-weak},
  \begin{equation}\label{eq:large-influence-bound}
    \Prob*{ \abs*{ p(U) -a } \leq \varepsilon } \lesssim \paren*{\frac{\varepsilon}{\eta}}^{\theta} + \paren*{\frac{\eta}{\norm{S_{i}}_{L^{2}}}}^{\theta_{d}}.
  \end{equation}

  \paragraph{Small influence estimate}
  The above bound is useless, whenever $\tau \coloneq \sup_{i} \Esp{ S_{i}^{2} }$ is small.
  In that case, we use a celebrated invariance principle for polynomials by \citeauthor{MDOInvariance}.
  Consider $G = (G_{i})$ a vector of independent standard Gaussian variables, then according to \cite[Thm.\ 2.1]{MDOInvariance},
  \begin{equation*}
    \Prob*{ \abs*{ p(U) - a } \leq \varepsilon } \lesssim d \tau^{1/8(d+1)} + \Prob*{ \abs*{p(G) - a} \leq \varepsilon }.
\end{equation*}
By a famous inequality of \citeauthor{CarberyWright} \cite[Cor.\ of Thm.\ 2]{CarberyWright}, we find that
\begin{equation*}
  \Prob*{ \abs*{ p(G) - a} \leq \varepsilon } \leq \paren*{\frac{\varepsilon}{\norm{p(G) - a}_{L^{2}}}}^{1/(d+1)}.
\end{equation*}
By assumption, we have that $\Var*{p(G)} = 1$.
Altogether this shows that
\begin{equation}\label{eq:small-influence-bound}
  \Prob*{ \abs*{p(U) - a} \leq \varepsilon } \lesssim \tau^{1/8(d+1)} + \varepsilon^{1/(d+1)}.
\end{equation}
\paragraph{Combining the estimates}
Whenever $\tau \lesssim \varepsilon^{1/4}$, we choose \cref{eq:small-influence-bound} yielding
\begin{equation*}
  \Prob*{ \abs*{ p(U) -a } \leq \varepsilon } \lesssim \varepsilon^{1/32(d+1)} + \varepsilon^{1/4(d+1)}.
\end{equation*}
On the other hand, when $\tau \gtrsim \varepsilon^{1/4}$, we choose \cref{eq:large-influence-bound} with $\eta = \varepsilon^{1/2}$, and this gives
\begin{equation*}
  \Prob*{ \abs*{p(U) - a} \leq \varepsilon } \lesssim \varepsilon^{\theta/2} + \varepsilon^{\theta_{d}/4}.
\end{equation*}
Thus we have that
\begin{align*}
  & \theta_{1} = \frac{1}{32} \wedge \frac{\theta}{2},
\\& \theta_{d+1} = \frac{\theta_{d}}{4}, \qquad d > 1.
\end{align*}
This shows that $\theta_{d} = (\frac{1}{32} \wedge \frac{\theta}{2})/4^{d-1}$.

\end{proof}

By using the comparison between $L^{1}$ and $L^{2}$, we immediately get the following corollary.
\begin{corollary}\label{th:small-ball-P-l1}
  Under the same assumptions as in \cref{th:small-ball-l2}, and assume moreover that $U_{1} \geq 0$, and $a_{I} \geq 0$.
  Then,
  \begin{equation}\label{eq:small-ball-P-l1}
    \sup_{a \in \mathbb{R}} \Prob*{ \abs{p(U) - a} \leq \varepsilon } \lesssim \paren*{\frac{\varepsilon}{\Esp*{ p(U)}}}^{\theta_{d}}.
  \end{equation}
\end{corollary}

\subsubsection{Improving positivity by splitting independent terms}

Recall that we have fixed a symmetric Hilbert--Schmidt operator $\mathsf{A}$ with vanishing diagonal, and with unit norm.
Fix $q \in \mathbb{N}$.
Recall that we consider the set of subsets of $\mathbb{N}$ with \emph{exactly} $q$ elements, that is
\begin{equation*}
  \mathscr{P}_{q}(\mathbb{N}) \coloneq \set*{ I \subset \mathbb{N} : \abs{I} = q }.
\end{equation*}
We consider:
\begin{itemize}
  \item the symmetric trace-class operator
    \begin{equation*}
      \mathsf{B} \coloneq \paren*{ b_{IJ} : I,\, J \in \mathscr{P}_{q}(\mathbb{N})},
    \end{equation*}
    with $b_{IJ} \coloneq \paren*{\det \mathsf{A}(I,J)}^{2} \geq 0$, for all $I$ and $J \in \mathscr{P}_{q}(\mathbb{N})$;
  \item a sequence of independent non-negative random variables $(U_{i} : i \in \mathbb{N})$, and
    \begin{equation*}
      U_{I} \coloneq \prod_{i \in I} U_{i}, \qquad I \in \mathscr{P}_{q}(\mathbb{N});
    \end{equation*}
  \item and the non-negative random variable:
    \begin{equation*}
      S \coloneq \sum_{I,J \in \mathscr{P}_{q}(\mathbb{N})} b_{IJ} U_{I} U_{J} 1_{I \cap J = \emptyset}.
    \end{equation*}
\end{itemize}

Recall that we have defined the $\ell^{1}$-influence $\upsilon(\mathsf{B})$ and the total mass $\sigma(\mathsf{B})$ in \cref{s:linear-algebra-ell-1}.
In the next section, we apply the results of this section to $U_{i} = \Gamma_{i}$.
In this case, one would have $S = \mathcal{S}_{q}(F)$ as defined in \cref{eq:sq}.
However, the results of this section only depend on the fact that the $\Gamma_{i}$'s are non-negative, and they depend on no other properties of the carré du champ.

\begin{remark}
  The random variable $S$ is non-negative since all the terms are non-negative.
  However, it is \emph{not} non-negative as a quadratic form in $U_{I}$ since the operator $\mathsf{B}$ may fail to be non-negative.
\end{remark}

\begin{proposition}\label{th:small-ball-improved}
  Assume that $\Esp*{ U_{1}^{2} } < \infty$, and that there exists $\theta > 0$ such that
  \begin{equation}\label{ass:mini-small-ball-U}
    \Prob*{ U_{1} \leq \varepsilon } \lesssim \varepsilon^{\theta}, \qquad \varepsilon > 0.
  \end{equation}
  Let $\kappa \in \mathbb{N}$.
  Assume that
  \begin{equation*}
    \tau(\mathsf{A}) \leq \frac{1}{(4q)^{5\kappa+2}}.
  \end{equation*}
  Then, with $\theta_{2q}$ defined in \cref{eq:def-theta},
  \begin{equation*}
    \Prob*{ S \leq \varepsilon} \lesssim \paren*{\frac{\varepsilon 32^{\kappa}}{\mathcal{R}_{q}(\mathsf{A})  - \tau(\mathsf{A}) \mathcal{R}_{q-2}(\mathsf{A}) }}^{\theta_{2q} 2^{\kappa}}, \qquad \varepsilon > 0.
  \end{equation*}
\end{proposition}

  \paragraph{Heuristics}
  Consider disjoint subsets $\mathbb{L}_{l} \subset \mathscr{P}_{q}(\mathbb{N})$, for $l = 1, \dots, 2^{\kappa}$.
Let
\begin{equation*}
  S_{l} \coloneq \sum_{(I,J) \in \mathbb{L}_{l} \times \mathbb{L}_{l}} 1_{I \cap J = \emptyset} b_{IJ} U_{I} U_{J}.
\end{equation*}
Then, the $S_{l}$'s are mutually independent.
Since each term in the sum on the right-hand side is a multi-linear polynomial of degree $2q$, provided we can apply \cref{eq:small-ball-P-l1} on each of the first term, this allows to improve our bound.
We now show this procedure is indeed feasible.

\paragraph{Splitting the mass via a probabilistic method}
Given $\mathbb{L} \subset \mathscr{P}_{q}(\mathbb{N})$, we write $\mathsf{B}^{\mathbb{L}} = (b_{IJ} : I,\, J \in \mathbb{L})$ for the extracted operator.
\begin{lemma}\label{th:mass-splitting}
  There exists $\mathbb{L} \subset \mathscr{P}_{q}(\mathbb{N})$ such that
  \begin{align*}\label{eq:total-mass-after-splitting}
    &\sigma\paren*{\mathsf{B}^{\mathbb{L}}} \geq \frac{1}{16} \sigma(\mathsf{B}) - \frac{2q}{16} \upsilon(\mathsf{B});
  \\&\sigma\paren*{\mathsf{B}^{\mathbb{L}^{C}}} \geq \frac{1}{16} \sigma(\mathsf{B}) - \frac{2q}{16} \upsilon(\mathsf{B}).
  \end{align*}
\end{lemma}

\begin{proof}
  Our proof implements the probabilistic method.
  Since $\sigma$ and $\upsilon$ are linear, by homogeneity, we assume that $\sigma(\mathsf{B}) = 1$.
  Let us consider a family of independent Bernoulli variables $(\varepsilon_{I} : I \in \mathscr{P}_{q}(\mathbb{N}))$ with mean $1/2$.
  We define the random set and random variables
  \begin{align*}
    & \mathbb{L} \coloneq \set*{ I \in \mathscr{P}_{q}(\mathbb{N}) : \varepsilon_{I} = 1 };
  \\& T \coloneq \sum_{(I,J) \in \mathbb{L} \times \mathbb{L}} b_{IJ} = \sum_{IJ} \varepsilon_{I} \varepsilon_{J} b_{IJ}.
  \\& \hat{T} \coloneq \sum_{(I,J) \in \mathbb{L}^{C} \times \mathbb{L}^{C}} b_{IJ} = \sum_{IJ} (1-\varepsilon_{I}) (1-\varepsilon_{J}) b_{IJ}.
  \end{align*}
  For $I,J,I',J' \in \mathscr{P}_{q}(\mathbb{N})$, we have that if $\varepsilon_{I} \varepsilon_{J} (1-\varepsilon_{I'})(1-\varepsilon_{J'}) \ne 0$, then $\varepsilon_{I} = \varepsilon_{J} = 1- \varepsilon_{I'} = 1 - \varepsilon_{J'} = 1$.
  This implies that $\set{I,J} \cap \set{I', J'} = \emptyset$.
  This later condition is implied by $(I \cup J) \cap (I' \cup J') = \emptyset$.
  By independence, for such $I,J,I',J'$, we have
  \begin{equation*}
    \Esp*{ \varepsilon_{I} \varepsilon_{J} (1-\varepsilon_{I'})(1-\varepsilon_{J'}) } = \Esp*{\varepsilon_{I} \varepsilon_{J} } \Esp*{(1-\varepsilon_{I'})(1-\varepsilon_{J'})}.
  \end{equation*}
  By a direct computation, we have
  \begin{equation*}
    \Esp*{ \varepsilon_{I} \varepsilon_{J}} = \frac{1}{4} 1_{I \ne J} + \frac{1}{2} 1_{I=J} \geq \frac{1}{4}.
\end{equation*}
  Developing the product we thus find
  \begin{equation*}
    \Esp*{T \hat{T}} \geq \frac{1}{16} \sum_{(I \cup J) \cap (I' \cup J') = \emptyset}  b_{IJ} b_{I'J'} = \frac{1}{16} (\sigma(\mathsf{B}))^{2} - \frac{1}{16} \sum_{(I \cup J) \cap (I' \cup J') \ne \emptyset} b_{IJ} b_{I'J'}.
  \end{equation*}
  Now we compute
  \begin{equation*}
    \sum_{(I \cup J) \cap (I' \cup J') \ne \emptyset} b_{IJ} b_{I'J'} \leq \sum_{i \in \mathbb{N}} \paren*{\sum_{I,J} 1_{i \in I \cup J} b_{IJ}}^{2} = \sum_{i \in \mathbb{N}} \upsilon_{i}(\mathsf{B})^{2} \leq \upsilon(\mathsf{B}) \sum_{i \in \mathbb{N}} \upsilon_{i}(\mathsf{B}) \leq 2q \upsilon(\mathsf{B}) \sigma(\mathsf{B}).
  \end{equation*}
  For the last inequality, we have used 
  \begin{equation*}
    \sum_{i \in \mathbb{N}} 1_{i \in I \cup J} = \abs{I \cup J} \leq 2q.
  \end{equation*}
  Since $\sigma(\mathsf{B}) = 1$, we have shown
  \begin{equation*}
    \Esp*{ T \hat{T}} \geq \frac{1}{16} (1 - 2q\upsilon(\mathsf{B})).
  \end{equation*}
  Owing to the fact that $T \vee \hat{T} \leq 1$, we find
  \begin{equation*}
    \Esp*{ T \wedge \hat{T} } \geq \Esp*{T \hat{T}} \geq \frac{1}{16} (1 - 2q \upsilon(\mathsf{B})).
  \end{equation*}
  This shows that there exists a realization of $\mathbb{L}(\omega)$ fulfilling the requirements of the lemma.
\end{proof}

\paragraph{Reduction to multi-linear polynomials}
By definition, $\sigma(\mathsf{B}) = \mathcal{R}_{q}(\mathsf{A})$.
However, in the definition of $S \coloneq \sum_{(I,J) \in \mathbb{L} \times \mathbb{L}} 1_{I \cap J = \emptyset} b_{IJ} U_{I} U_{J}$, we have removed the \enquote{diagonal} in order to work with multi-linear polynomials.
The following result allows to verify that we do not lose too much mass.

\begin{lemma}\label{th:multi-linear-no-diagonal}
  Assume that $b_{IJ} = \paren*{\det \mathsf{A}(I,J)}^{2}$, with $\mathsf{A}$ as above.
  Let $\mathbb{L} \subset \mathscr{P}_{q}(\mathbb{N})$.
  Define $\hat{\mathsf{B}}^{\mathbb{L}} \coloneq (b_{IJ} : I,\, J \in \mathbb{L},\, I \cap J = \emptyset)$.
  Then, 
  \begin{equation*}
    \abs*{\sigma(\mathsf{B}^{\mathbb{L}}) - \sigma(\mathsf{\hat{B}}^{\mathbb{L}})} \lesssim \tau(\mathsf{A}) \mathcal{R}_{q-2}(\mathsf{A}).
  \end{equation*}
\end{lemma}

\begin{proof}
  To lighten the notation, we drop the dependence in $\mathbb{L}$.
  If $q =1$, the claim is
  \begin{equation*}
    \sum_{i,j \in \mathbb{L}} a_{ij}^{2} = \sum_{i,j \in \mathbb{L}} 1_{i \ne j} a_{ij}^{2},
  \end{equation*}
  which is immediate, since by assumption $\mathsf{A}$ vanishes on the diagonal.
  Thus, we assume that $q \geq 2$.
  We have
  \begin{equation*}
    \sigma(\mathsf{B}) = \sigma(\widehat{\mathsf{B}}) + \sum_{\substack{\abs{I} = \abs{J} = q\\I,J \in \mathbb{L}}} 1_{I \cap J \ne \emptyset} b_{IJ}.
  \end{equation*}
  Let $I = \{i_{1}, \dots, i_{q}\}$ and $J = \{j_{1}, \dots, j_{q}\} \in \mathbb{L}$.
  Assume that $m \in I \cap J$.
  By assumption there exist $k_{0}$ and $l_{0} \in \{1,\dots, q\}$ such that $i_{k_{0}} = j_{l_{0}} = m$.
  Since by assumption, $a_{mm} = 0$, expanding the determinant along the row $m$ gives
  \begin{equation*}
    \det \mathsf{A}(I,J) = \sum_{l\ne l_{0}} (-1)^{l+m} a_{m, j_{l}} \det \mathsf{A}(I \setminus \{m\}, J \setminus \{j_{l}\}).
  \end{equation*}
  Since for $l \ne l_{0}$, $m \in J \setminus \{j_{l}\}$, we expand the along the row $m$, and we get, by symmetry of $\mathsf{A}$
  \begin{equation*}
    \det \mathsf{A}(I,J) = \sum_{l\ne l_{0}} \sum_{k \ne k_{0}} (-1)^{l+k+2m-1} a_{m, j_{l}} a_{m, i_{k}} \det \mathsf{A}(I \setminus \{m, i_{k}\}, J \setminus \{m, j_{l}\}).
  \end{equation*}
  When $q =2$, in the above expression $\det \mathsf{A}(\emptyset,\emptyset)$ is understood as $1$.
  Thus,
  \begin{equation*}
    \begin{split}
      \sum_{\abs{I} = \abs{J} = q} 1_{I \cap J \ne \emptyset} b_{IJ} & \lesssim \sum_{m \in \mathbb{N}} \sum_{\abs{I} = \abs{J} = q} \sum_{j \in J \setminus \{m\}} \sum_{i \in I \setminus \{m\}} a_{m,j}^{2} a_{m,i}^{2} \paren*{\det \mathsf{A}(I \setminus \{m,i\}, J \setminus \{m,j\})}^{2}.
                                                                   \\& \lesssim \sum_{m \in \mathbb{N}} \tau_{m}(\mathsf{A})^{2} \sum_{\abs{I} = \abs{J} = q-2} \paren*{ \det \mathsf{A}(I,J)}^{2}
                                                                   \\& \lesssim \tau(\mathsf{A}) \mathcal{R}_{1}(\mathsf{A}) \mathcal{R}_{q-2}(\mathsf{A}).
    \end{split}
  \end{equation*}
  This concludes the proof.
\end{proof}

\paragraph{Extraction of independent sums}

We now have all the necessary tools to conclude.
\begin{proof}[Proof of \cref{th:small-ball-improved}]
Let us define for all $k \in \mathbb{N}$ the pairwise disjoint sets $\mathbb{L}^{(k)}_{1}, \dots, \mathbb{L}^{(k)}_{2^{k}} \subset \mathscr{P}_{q}(\mathbb{N})$ using an iterative procedure described below.
To each of the $\mathbb{L}_{l}^{(k)}$, we associate the operators $\mathsf{B}^{(k)}_{l} \coloneq \mathsf{B}^{\mathbb{L}^{(k)}_{l}}$, and the random sums
\begin{equation*}
  S_{l}^{(k)} \coloneq \sum_{(I,J) \in \mathbb{L}^{(k)}_{l} \times \mathbb{L}^{(k)}_{l}} b_{IJ} U_{I} U_{J}.
\end{equation*}
Since the $U_{i}$'s are independent and that, for a given $k$, the $\mathbb{L}^{(k)}_{l}$'s are pairwise disjoint, the $S_{l}^{(k)}$ are mutually independent.

We initialize $\mathbb{L}^{(0)}_{1} \coloneq \mathscr{P}_{q}(\mathbb{N})$.
Then, for all $k \in \mathbb{N}$, if there exists $l \in \set{1, \dots, 2^{k}}$ such that $2q\upsilon(\mathsf{B}) > \frac{1}{2} \sigma(\mathsf{B}^{(k)}_{l})$, we stop the procedure; otherwise define $\mathbb{L}^{(k+1)}_{l}$ and $\mathbb{L}^{(k+1)}_{l + 2^{k}}$ as the sets obtained by applying \cref{th:mass-splitting} to $\mathsf{B}^{(k)}_{l}$.
We write $k_{max}$ the integer at which the procedure stops.
Our procedure guarantees that
\begin{equation}\label{eq:small-ball-improved-large-mass}
  \sigma(\mathsf{B}^{(k)}_{l}) \geq \frac{\sigma(\mathsf{B})}{32^{k}}, \qquad k \leq k_{max}.
\end{equation}
On the other hand at step $k$ the procedure continues provided
\begin{equation*}
  \upsilon(\mathsf{B}) \leq \frac{1}{(4q)^{5k + 1}}.
\end{equation*}
This shows that $k_{max} \geq  \kappa$.
Using the mutual independence of the $S^{(k_{max})}_{l}$'s and the fact that the $b_{IJ}$'s and $U_{I}$'s are non-negative we get
\begin{equation}\label{eq:small-ball-improved-S}
  \Prob*{ S \leq \varepsilon } \leq \prod_{l=1}^{2^{\kappa}} \Prob*{ S^{(\kappa)}_{l} \leq \varepsilon }.
\end{equation}
Define
\begin{equation*}
  \hat{\mathsf{B}}_{l}^{(\kappa)} \coloneq (b_{IJ} 1_{\{ I \cap J = \emptyset \}} : I,J \in \mathbb{L}_{l}^{(\kappa)});
\end{equation*}
By invoking \cref{th:small-ball-P-l1}, and \cref{th:multi-linear-no-diagonal} together with \cref{eq:small-ball-improved-large-mass}, we find that, since $\mathcal{R}_{q-2}(\mathsf{A}) \leq \mathcal{R}_{1}(\mathsf{A}) = 1$,
\begin{equation}\label{eq:small-ball-improved-S-hat}
  \Prob*{ S_{l}^{(\kappa)} \leq \varepsilon } \lesssim \paren*{ \frac{\varepsilon}{\Esp{S_{l}^{(\kappa)}}}}^{\theta_{2q}} \lesssim \paren*{ \frac{\varepsilon}{\sigma(\hat{\mathsf{B}}_{l}^{(\kappa)})}}^{\theta_{2q}} \lesssim \paren*{ \frac{32^{\kappa} \varepsilon}{\sigma(\mathsf{B}) - 32^{\kappa}\tau(\mathsf{A})}}^{\theta_{2q}}.
\end{equation}
Combining \cref{eq:small-ball-improved-S,eq:small-ball-improved-S-hat}, the proof is complete, since $\sigma(\mathsf{B}) = \mathcal{R}_{q}(\mathsf{A})$
\end{proof}

\subsubsection{Control of the spectral remainders of \texorpdfstring{$\Gamma^{1/2} \mathsf{A} \Gamma^{1/2}$}{Γ1/2 A Γ1/2}}

\begin{proof}[{Proof of \cref{th:regularity-quadratic-form}}]
  Fix $\mathsf{A}$ satisfying our assumptions.
  For short write $Q \coloneq Q_{\mathsf{A}}$.
  The general strategy is to apply \cref{th:regularity-negative-moments-spectral-remainder}.
   We obtain
  \begin{equation*}
    \mathbf{W}_{q,1}(Q) \leq \boldsymbol{\Psi}_{q}(Q) \boldsymbol{\Psi}_{8q+1}(\sharp_{G}Q) \Esp*{ \mathcal{S}_{128q+18}(Q)^{-1/4}}^{1/4}.
  \end{equation*}
  By \cref{th:polynomial:equivalence} with $\mathcal{X} \coloneq \{X_{1}\}$ and $m \coloneq 2$, since $\norm{Q}_{L^{2}} = 1$, in view of the recursive definition of the $\Psi_{k}$'s, using repeatedly Hölder's inequality, there exists $C_{q} > 0$ such that
  \begin{equation*}
    \boldsymbol{\Psi}_{q}(Q) \leq C_{q}, \qquad q \in \mathbb{N}.
  \end{equation*}
  Similarly, $\Esp*{(\sharp_{G} Q)^{2}} = \Esp*{\Gamma(Q, Q)}$ is some fixed constant and
  \begin{equation*}
    \sharp_{G} Q = 2 \sum_{ijk} a_{ij} X_{i} (\mathsf{D} X_{j} \cdot \vec{G}_{j}).
  \end{equation*}
  Thus, applying \cref{th:polynomial:equivalence} with $\mathcal{X} \coloneq \set*{ X_{1}, (\mathsf{D} X_{1} \cdot \vec{G}_{1})}$ and $m \coloneq 2$, we conclude that there exists $C'_{q} > 0$ such that
  \begin{equation*}
    \boldsymbol{\Psi}_{q}(\sharp_{G}Q) \leq C'_{q}, \qquad q \in \mathbb{N}.
\end{equation*}
Thus we are left to control the spectral remainders of $\mathcal{S}_{128q+18}(Q)$.

Let $q' = 128q+18$, and $\kappa \in \mathbb{N}$ such that $2^{\kappa} \theta_{2q'} > 1/4$.
Let us take $\tau_{q,\delta} \coloneq (4q)^{-5\kappa -2} \wedge \frac{\delta}{2}$.
  Recall that we have set $b_{IJ} \coloneq (\det \mathsf{A}(I,J))^{2}$ for $I$ and $J \in \mathscr{P}_{q'}(\mathbb{N})$.
  Since $\nabla^{2} Q = \mathsf{A}$,
  \begin{equation*}
    \mathcal{S}_{q'}(Q) = \sum_{|I|=|J|=q'} b_{IJ} \Gamma_{I} \Gamma_{J} 1_{I \cap J = \emptyset}.
  \end{equation*}
  By \cref{eq:small-ball-theta} and choice of $\kappa$, the assumptions of \cref{th:small-ball-improved} are fulfilled, and, since $\tau(\mathsf{A}) \leq \tau_{q,\delta} \leq \delta/2$ and $\mathcal{R}_{q'-2}(\mathsf{A}) \leq \mathcal{R}_{1}(\mathsf{A}) = 1$, we find that
    \begin{equation}\label{eq:explicit-bound-spectral-remainder-hessian}
      \begin{split}
        \Prob*{ \mathcal{S}_{q'}(Q) \leq \varepsilon } &\lesssim \paren*{ \frac{\varepsilon 32^{\kappa}}{\mathcal{R}_{q'}(\mathsf{A}) - \tau(\mathsf{A}) \mathcal{R}_{q'-2}(\mathsf{A})} }^{\theta_{2q'} 2^{\kappa}} 
                                                     \\&\lesssim \paren*{ \frac{\varepsilon 32^{\kappa}}{\delta/2} }^{\theta_{2q'} 2^{\kappa}} 
      \end{split}
    \end{equation}
  Whenever $\eta_{q} \coloneq 2^{\kappa} \theta_{2q'} > 1/4$, this shows that
  \begin{equation*}
    \Esp*{ \mathcal{S}_{q'}(Q)^{-1/4}}^{1/4} \lesssim \delta^{-\eta_{q}}.
  \end{equation*}
  We conclude by \cref{th:regularity-negative-moments-spectral-remainder}.
\end{proof}

\begin{remark}[Improved constants when $X_{1}$ is log-concave]\label{rk:improved-log-concave}
In the previous proof, we have use \cref{ass:small-ball-gamma} combined with a splitting argument in order to derive a small ball estimate for
\begin{equation*}
  S_{q} \coloneq \sum_{I,J \in \mathscr{P}_{q}(\mathbb{N})} b_{IJ} \Gamma_{I} \Gamma_{J}.
\end{equation*}
Since we rely on \cref{th:small-ball-l2}, we obtain the non-optimal exponent $\theta_{d}$.
However, in some cases we can expect a much better exponent to appear.
This is for instance the case as soon as $S_{q}$ is a multi-linear polynomial of degree $d_{q}$ in the $X_{i}$ and $X_{1}$ as a log-concave law.
In this case by an inequality of \citeauthor{CarberyWright} \cite[Cor.\ of Thm.\ 2]{CarberyWright}, we find that
\begin{equation*}
  \Prob*{  S_{q} \leq \varepsilon } \lesssim \paren*{ \frac{\varepsilon}{ \Esp{S_{q}}} }^{\frac{1}{d_{q}}}.
\end{equation*}
We could redo the same proof with $\theta_{d}$ replaced by $1/d$.
This could yield to better constants.
\end{remark}

\section{Examples and applications}

\subsection{Examples of laws satisfying our assumptions}\label{s:examples}
\newcommand{\parexample}{\paragraph}
We now give concrete examples where \cref{th:regularity-quadratic-form} can be applied.
Namely, we give explicit laws $\mu$ that can be realized as the law of a random variable $X \in \mathbb{D}^{\infty}$ such that $\Gamma(X,X)$ has some small negative moment.
Let us write $\mathscr{L}$ for the set of all such laws.

\subsubsection{Explicit laws}
Our general theorem to build examples of laws in $\mathscr{L}$ is as follows.
We write $\Phi$ for the cumulative distribution function of the standard Gaussian distribution, namely
\begin{equation*}
\Phi(x) \coloneq \int_{-\infty}^{x} \mathrm{e}^{-s^{2}/2} \frac{\mathrm{d}s}{\sqrt{2\pi}};
\end{equation*}

\begin{theorem}\label{th:examples:general}
  Let $\mu \in \mathscr{P}(\mathbb{R})$ supported on some interval $I$.
  Write $q$ for its quantile function, and let $X \coloneq q(\Phi(N))$, where $N$ is a standard Gaussian.
  \begin{enumerate}[(i)]
    \item\label{i:examples:general:smooth} Assume that $\mu$ has finite $p$-moments for all $p \in \mathbb{N}$, and $q$ is a $\mathscr{C}^{\infty}$-diffeormorphsim from $(0,1)$ to $I$ and that
    \begin{align}
      & \label{ass:quantiles:equivalent:0} \limsup_{t \to 0^{-}} t^{k} q^{(k)}(t) < \infty;
    \\& \label{ass:quantiles:equivalent:1} \limsup_{t \to 1^{+}} (1-t)^{k} q^{(k)}(t) < \infty.
    \end{align}
    Then, $X \in \mathbb{D}^{\infty}$.
    \item\label{i:examples:general:small-ball} Assume that $q \in \mathscr{C}^{1}$ and that $\mu$ admits a density in $L^{1+\varepsilon}$ fopr some $\varepsilon > 0$.
    Then, $X$ statisfies \cref{eq:small-ball-theta} for some $\theta > 0$. 
\end{enumerate}
In particular, if both conditions hold $\mu \in \mathscr{L}$.
\end{theorem}

\begin{proof}
  \cref{i:examples:general:smooth}
  Write $T = q \circ \Phi$.
  It is well-known that $T$ transport the Gaussian distribution to $\mu$.
Since $\Phi \colon \mathbb{R} \to (0,1)$ is a $\mathscr{C}^{\infty}$-diffeomorpshim, and $q \colon (0,1) \to I$ is also a $\mathscr{C}^{\infty}$-diffeomorphism, we have that $T$ is a smooth diffeomorphism.
To check that $X \in \mathbb{D}^{\infty}$, it is sufficient to check that $T^{(k)}(N) \in L^{p}$ for all $k \in \mathbb{N}^{*}$ and $p \geq 1$.
Classically,
\begin{equation*}
  \Phi^{(k+1)}(x) = (-1)^{k} \mathrm{e}^{-x^{2}/2} H_{k}(x), \qquad x \in \mathbb{R},
\end{equation*}
where $H_{k}$ is the $k$-th Hermite polynomial.
For $k \in \mathbb{N}^{*}$, write $\mathfrak{P}_{k}$ for the set of all partitions of $\{1,\dots,k\}$.
By the Faà di Bruno formula, we have
\begin{equation*}
  T^{(k)} = \sum_{\tau \in \mathfrak{P}_{k}} q^{(|\tau|)} \circ \Phi \prod_{I \in \tau} \Phi^{(|I|)} = \mathrm{e}^{-kx^{2}/2} \sum_{\tau \in \mathfrak{P}_{k}} q^{(|\tau|)} \circ \Phi \prod_{I \in \tau} (-1)^{|I|-1} H_{|I|-1}.
\end{equation*}
We recall that as $x \to \infty$, $\Phi(x) \sim 1 - c \mathrm{e}^{-x^{2}/2} \frac{1}{x}$.
Thus by \cref{ass:quantiles:equivalent:0}, we find that
\begin{equation*}
  \abs{T^{(k)}(x)} \lesssim \mathrm{e}^{-kx^{2}/2} \sum_{\tau \in \mathfrak{P}_{k}} \paren*{\frac{1}{x} \mathrm{e}^{-x^{2}/2}}^{- |\tau|} \prod_{I \in \tau} H_{|I|-1}(x), \qquad x \to \infty.
\end{equation*}
We see that the worst behaviour as $x \to \infty$ arises when $|\tau| = k$.
In this case the equivalent gives that $T^{(k)}$ has a polynomial behaviour as $x \to \infty$.
A similar argument  based on \cref{ass:quantiles:equivalent:0} shows that $T^{(k)}$ has a polynomial behaviour as $x \to -\infty$.
This shows that
\begin{equation*}
  \Esp*{\abs{T^{(k)}(N)}^{p}} < \infty, \qquad k \in \mathbb{N}^{*},\, p \geq 1,
\end{equation*}
which was to be demonstrated.

\cref{i:examples:general:small-ball}
In view of the assumption, with $F$ the cumulative distribution function of $\mu$
\begin{equation*}
  \Gamma(T(N), T(N)) = q'(\Phi(N))^{2} (\Phi'(N))^{2} = \frac{1}{F'(q(\Phi(N)))^{2}} \mathrm{e}^{-N^{2}}.
\end{equation*}
Using that $X = q \circ \Phi(N)$, the find
\begin{equation*}
  \Esp*{ \Gamma(X,X)^{-\theta} } = \Esp*{ \mathrm{e}^{\theta N^{2}} F'(X)^{2\theta}}.
\end{equation*}
By Cauchy--Schwarz inequality and assumption the above quantity is finite for $\theta$ small enough.
The proof is complete.
\end{proof}

We now compute explicit examples.

\parexample{Beta distribution}\label{s:examples:beta}
Over $[-1,+1]$, we consider $\beta(a,b)$ the Beta distribution with parameter $a > -1$ and $b > -1$, that is the law whose density with respect to the Lebesgue measure is proportional to $\paren*{1-x}^{a} \paren*{1+x}^{b}$.

\begin{theorem}\label{th:examples:beta}
  For all $a > -1$ and $b > -1$, $\beta(a,b) \in \mathscr{L}$.
\end{theorem}

\begin{proof}
  The density distribution of $\beta(a,b)$ is $\mathscr{C}^{\infty}((-1,1))$.
  Thus the cumulative distribution function is a $\mathscr{C}^{\infty}$-diffeomorphism, and thus so is the quantile function $q$.
  Moreover, $\beta(a,b)$ compactly supported thus \cref{i:examples:general:small-ball} in \cref{th:examples:general} holds.
  We very that \cref{i:examples:general:smooth} also holds thanks to \cref{th:beta:quantiles} below.
\end{proof}

\begin{lemma}\label{th:beta:quantiles}
  Let $a > -1$, $b > -1$, and $q$ the quantile function of $\beta(a,b)$.
  Then for all $k \in \mathbb{N}$
  \begin{align}
 & \text{as}\ t \to 0, \quad
    \begin{cases}
    & q(t) \sim  t^{\frac{1}{b+1}} - 1.
  \\& q^{(k)}(t) \sim t^{\frac{1}{b+1}-k}.
    \end{cases}
    \\
 & \text{as}\ t \to 1, \quad
    \begin{cases}
    & q(t) \sim 1 - (1-t)^{\frac{1}{a+1}}.
  \\& q^{(k)}(t) \sim (1-t)^{\frac{1}{a+1}-k}.
    \end{cases}
  \end{align}
\end{lemma}

\begin{proof}
  By symmetry we only establish the equivalents at $t \to 1$.
  In this proof the constant $c > 0$ can change from line to line.
  Consider the cumulative distribution function of $\beta(a,b)$:
\begin{equation*}
  F(x) \coloneq c \int_{-1}^{x} (1-s)^{a} (1+s)^{b} \mathrm{d} s, \qquad x \in (-1,1).
\end{equation*}
Near $1$ we find
\begin{equation*}
  F(x) \sim 1 - c \int_{x}^{1} (1-s)^{a} = 1 - c (1-x)^{a+1}.
\end{equation*}
Using that $F \circ q = \mathsf{id}$, we thus get
\begin{equation*}
  1 - c (1- q(t))^{a+1} \sim t,
\end{equation*}
and from there
\begin{equation*}
  q(t) \sim 1 - c (1-t)^{\frac{1}{a+1}}.
\end{equation*}
This proves the claim for $k =0$.
Using the rule for derivation of inverse functions, we find that
\begin{equation}\label{eq:beta:equa-diff-q}
  q' = \frac{c}{(1-q)^{a} (1+q)^{b}}.
\end{equation}
Thus
\begin{equation*}
  q' \sim \frac{c}{(1-t)^{\frac{a}{a+1}}} = c (1-t)^{\frac{1}{a+1} - 1}.
\end{equation*}
This proves the claim for $k = 1$.
Let us now proceed by induction.
Applying the Leibniz rule and then the Faà di Bruno formula to \cref{eq:beta:equa-diff-q}, we find
\begin{equation*}
  q^{(k+1)} = \sum_{p=0}^{k} \binom{k}{p} \sum_{\tau \in \mathfrak{P}_{p}} (1-q)^{-a-|\tau|} \prod_{I \in \tau} q^{(|I|)} \sum_{\sigma \in \mathfrak{P}_{k-p}} (1+q)^{-b-|\sigma|} \prod_{J \in \sigma} q^{(|J|)}.
\end{equation*}
Applying the induction hypothesis we thus find that near $1$
\begin{equation*}
  q^{(k+1)} \sim c \sum_{p=0}^{k} \binom{k}{p} \sum_{\tau \in \mathfrak{P}_{p}} (1-t)^{- \frac{a +|\tau|}{a+1}} \prod_{I \in \tau} (1-t)^{\frac{1}{a+1}- |I|} \sum_{\sigma \in \mathfrak{P}_{k-p}} \prod_{J \in \sigma} (1-t)^{\frac{1}{a+1} - |J|}.
\end{equation*}
Using that for $\tau \in \mathfrak{P}_{p}$, $\sum_{I \in \tau} |I| = p$, we see that exponent in $(1-t)$ in one summand is
\begin{equation*}
  -\frac{a}{a+1} - \frac{|\tau|}{a+1} + \frac{|\tau|}{a+1} - p + \frac{|\sigma|}{a+1} - k +p = \frac{1}{a+1} - (k+1) + \frac{|\sigma|}{a+1}.
\end{equation*}
The worst exponent is thus obtained for $p = k$ in which case $\sigma = \emptyset$.
This gives the announced estimate.
\end{proof}

\parexample{Gamma distribution}\label{s:examples:Gamma}
Over $\mathbb{R}_{+}$, we consider $\gamma(a)$ the gamma distribution with parameter $a > -1$ with density proportional to $\mathrm{e}^{-x} x^{a}$.

\begin{theorem}\label{th:examples:gamma}
  For all $a > -1$, $\gamma(a) \in \mathscr{L}$.
\end{theorem}

\begin{proof}
  The density distribution of $\gamma(a)$ is $\mathscr{C}^{\infty}(\mathbb{R}_{+})$.
  Thus the cumulative distribution function is a $\mathscr{C}^{\infty}$-diffeomorphism, and thus so is the quantile function $q$.
  Moreover, there exists $\varepsilon > 0$ such that $a(1+\varepsilon) > -1$.
  Thus
  \begin{equation*}
    \int_{0}^{\infty} (\mathrm{e}^{-x} x^{a})^{1+\varepsilon} < \infty.
  \end{equation*}
  This shows that \cref{i:examples:general:small-ball} in \cref{th:examples:general} holds.
  We verify that \cref{i:examples:general:smooth} also holds thanks to \cref{th:gamma:quantiles} below.
\end{proof}

\begin{remark}
  Some Gamma distributions can be realized as elements of $\mathbb{D}^{\infty}$ without transporting them from a Gaussian via the quantile functions.
  For instance, for $k \in \mathbb{N}^{*}$, $\gamma(n/2-1)$ is a chi-squared law with $n$ degree of freedom and can thus be realised as a sum of square of Gaussian variables.
  We however give a proof based on quantiles as it allows to cover all the cases.
\end{remark}

\begin{lemma}\label{th:gamma:quantiles}
  Let $q$ be the quantile function of $\mu_{a}$.
  Then,
  \begin{align}
    & \text{as}\ t \to 0, \quad q^{k}(t) \sim t^{\frac{1}{a+1} - k}.
  \\& \text{as}\ t \to 1, \quad q^{k}(t) \sim (1-t)^{-k}.
  \end{align}
\end{lemma}

\begin{proof}
Consider the cumulative distribution function of $\mu_{a}$
\begin{equation*}
  F(x) \coloneq \int_{0}^{x} s^{a} \mathrm{e}^{-s} \mathrm{d} s.
\end{equation*}
As $x \to \infty$, we have
\begin{equation*}
  F(x) \sim 1 - x^{a} \mathrm{e}^{-x}.
\end{equation*}
Using that $F \circ q = \mathsf{id}$, we find that
\begin{equation*}
  \mathrm{e}^{q} q^{-a} \sim (1-t)^{-1}.
\end{equation*}
Now we use that
\begin{equation*}
  q' = \frac{1}{\mathrm{e}^{-q}q^{a}} \sim (1-t)^{-1}.
\end{equation*}
We proceed by induction.
By the Leibniz rule and the Faà di Bruno formula,
\begin{equation*}
  q^{(k+1)} = \sum_{p=0}^{k} \binom{k}{p} \sum_{\tau \in \mathfrak{P}_{p}} \mathrm{e}^{q} \prod_{I \in \tau} q^{(|I|)} \sum_{\sigma \in \mathfrak{P}_{k-p}} q^{-a-|\sigma|} \prod_{J \in \sigma} q^{(|J|)}.
\end{equation*}
By the induction hypothesis, we obtain the equivalent
\begin{equation*}
  q^{(k+1)} \sim \mathrm{e}^{q} q^{-a} \sum_{p=0}^{k} (1-t)^{-p} \sum_{\sigma \in \mathfrak{P}_{k-p}} q^{-|\sigma|} (1-t)^{-k+p}.
\end{equation*}
Even if we do not know an equivalent of $q$ as $t \to 1$, we know hat $q(t) \to +\infty$ as $t \to 1$.
Thus the worst case in the above equivalent is obtained for $p = k$ and then $|\sigma|=0$.
This proves the claim by induction.
Similarly, as $x \to 0$, we have
\begin{equation*}
F(x) \sim \int_{0}^{x} s^{a} \mathrm{d} s = x^{a+1}.
\end{equation*}
This shows that $q \sim t^{\frac{1}{a+1}}$ and $q' \sim t^{\frac{1}{a+1} - 1}$.
By induction, we also show the rest of the claim
\end{proof}

\subsubsection{Stability of our assumption under probabilistic operations}
Actually $\mathscr{L}$ is stable under various natural probabilistic operations.
We write $\mu \ast \nu$ for the convolution of $\mu$ and $\nu \in \mathscr{P}(\mathbb{R})$, that is the law of $X + Y$ with $(X,Y) \sim \mu \otimes \nu$.
Similarly, we write $\mu \diamond \nu$ for the law of $XY$.

\paragraph{Stability under convolution}
\begin{lemma}
  If $\mu$ and $\nu \in \mathscr{L}$.
  Then $\mu \ast \nu \in \mathscr{L}$.
\end{lemma}

\begin{proof}
  Let $X$ and $Y \in \mathbb{D}^{\infty}$ satisfying \cref{ass:small-ball-gamma} with law $\mu$.
  Without loss of generality, we can assume that $X$ and $Y$ are defined on independent Wiener spaces.
  Since $\mathbb{D}^{\infty}$ is a linear space, $X + Y \in \mathbb{D}^{\infty}$.
  Moreover, since $X$ and $Y$ are defined on different Wiener spaces we have
  \begin{equation*}
    \Gamma(X+Y,X+Y) = \Gamma(X,X) + \Gamma(Y,Y).
  \end{equation*}
  Thus $X+Y$ satisfies \cref{ass:small-ball-gamma}.
\end{proof}

\begin{remark}\label{rk:generalized-convolution}
  We could actually consider even more general laws.
  Indeed, if $\mu \in \mathscr{L}$, and $\nu \in \mathscr{P}(\mathbb{R})$, then \cref{th:regularity-quadratic-form} applies for $X \sim \mu \ast \nu$.
  To see this, consider $X \in \mathbb{D}^{\infty}$ and satisfying \cref{ass:small-ball-gamma} with law $\mu$ defined on the Wiener space, and $Y$ with law $\nu$ defined on another probability space $(\Omega_{0}, \mathfrak{W}_{0}, \prob_{0})$.
  In general, $Y$ cannot be realised on a Wiener so our analysis does not apply verbatim.
However, we could equip $(\Omega_{0}, \prob_{0})$ with the trivial carré du champ $\Gamma_{0}(F, F) = 0$ for all random variables $F$ defined on $\Omega_{0}$.
Then we can equip the product space $\bar{\Omega} \coloneq (\mathbb{R}, \mathfrak{B}(\mathbb{R}), \gamma)^{\mathbb{N}} \otimes (\Omega_{0}, \mathfrak{W}_{0}, \prob_{0})$, with he product Dirichlet structure of $\Gamma$ and $\Gamma_{0}$ (see \cite[Chap.~V \S 2]{BouleauHirsch}).
In this case, $\bar{\Gamma}(X+Y, X+Y) = \Gamma(X,X) + \Gamma_{0}(Y,Y) = \Gamma(X,X)$ thus $X+Y$ satisfies the small ball estimate \cref{eq:small-ball-theta}.
Moreover, defining
\begin{equation*}
  \bar{\sharp}_{G}(X+Y) \coloneq \sharp_{G} X \overset{\law}{=} \bar{\Gamma}(X+Y,X+Y)^{1/2} N,
\end{equation*}
and $\bar{\mathbb{D}}^{\infty} \coloneq \mathbb{D}^{\infty} \otimes L^{2}(\prob_{0})$, our abstract result \cref{th:sobolev-regularity-iterated-gradient} extends to this setting.
Indeed, this result is solely based on the Fourier--Laplace identity \cref{eq:laplace-fourier}, the chain rule \cref{eq:chain-rule-carre-du-champ}, the integration by parts \cref{eq:leibniz-formula-3}, and the fact that $\bar{\sharp}_{G}(X+Y) \in \bar{\mathbb{D}}^{\infty}$.
All these ingredients still hold for this generalized setting.
\end{remark}

\paragraph{Stability under product}
\begin{lemma}
  Let $\mu$ and $\nu \in \mathscr{L}$.
  Assume, moreover, that there exists $\theta' > 0$ such that
  \begin{equation*}
    \int y^{-\theta'} \nu(\mathrm{d}y) < \infty.
  \end{equation*}
  Then $\mu \diamond \nu \in \mathscr{L}$.
\end{lemma}

\begin{proof}
  Let $X$ and $Y \in \mathbb{D}^{\infty}$ satisfying \cref{ass:small-ball-gamma}.
  Without loss of generality, we assume that $X$ and $Y$ are defined on independent Wiener spaces.
  Since $\mathbb{D}^{\infty}$ is an algebra, we find that $XY \in \mathbb{D}^{\infty}$.
  Moreover,
  \begin{equation*}
    \Gamma(XY,XY) = X^{2} \Gamma(Y,Y) + Y^{2} \Gamma(X,X) \geq Y^{2} \Gamma(X,X).
  \end{equation*}
  Thus $XY$ satisfies \cref{ass:small-ball-gamma}.
\end{proof}

\begin{remark}
  As in \cref{rk:generalized-convolution}, we could generalize the result above.
  If $\mu \in \mathscr{L}$ and $\nu \in \mathscr{P}(\mathbb{R})$ with some negative finite negative moments, then we can apply \cref{th:regularity-quadratic-form} whenever $X \sim \mu \diamond \nu$.
  As in \cref{rk:generalized-convolution}, we construct a Dirichlet form on the product that is trivial on the $Y$-component.
  In this case we find
  \begin{equation*}
    \bar{\Gamma}(XY,XY) = Y^{2} \Gamma(X,X),
  \end{equation*}
  and the small ball estimate \cref{eq:small-ball-theta} holds.
\end{remark}

\parexample{Multi-linear chaos}\label{s:examples:chaos}
The two operations above allow us to consider multi-linear polynomials.
For instance consider $(Z_{i})$ a sequence of independent random variables with (possible non identical) laws are taken among Gaussian, Beta, or Gamma distribution.
Then, $p(Z_{1}, \dots, Z_{n})$ where $p$ is a multi-linear polynomials has also a law in $\mathscr{L}$.

\subsection{Applications}

\subsubsection{Improving normal convergence for quadratic forms}\label{s:normal-convergence}

In this section, we consider a sequence of independent copies $(X_{i})$ of some $X \in \mathbb{D}^{\infty}$ satisfying \cref{ass:small-ball-gamma}.
We also consider a sequence of quadratic forms
\begin{equation*}
  Q_{n} \coloneq \sum_{i,j} a_{ij}^{(n)} X_{i} X_{j},
\end{equation*}
for some sequence $\set*{ \mathsf{A}^{(n)} \coloneq (a_{ij}^{(n)}) : n \in \mathbb{N} }$ of symmetric Hilbert--Schmidt operators with vanishing diagonal.
In this section, we assume that $(\mathsf{A}^{(n)})$ satisfies some conditions ensuring that $(Q_{n})$ converges in law to a Gaussian distribution, and we show that the convergence is automatically improved in $\mathscr{C}^{\infty}$-convergence.
By this, we mean that for all $q \in \mathbb{N}$, there exists $N_{q} \in \mathbb{N}$ such that for $n \geq N_{q}$, the law of $Q_{n}$ has a $\mathscr{C}^{q}$-density and that this density converges in $\mathscr{C}^{q}$ to that of a standard Gaussian.

\paragraph*{Leptokurtic random variables}

We say that a random variable $Z$ is \emph{leptokurtic} whenever $\Esp*{Z} = 0$, $\Esp*{Z^{2}} = 1$, and $\Esp*{Z^{4}} \geq \Esp*{N^{4}} = 3$ where $N$ is a standard Gaussian variable. 

\begin{proposition}
  With the above notation, assume moreover that $X_{1}$ is leptokurtic and that
\begin{equation*}
  Q_{n} \xrightarrow[n \to \infty]{law} \mathscr{N}(0,1).
\end{equation*}
Then,
\begin{equation*}
  Q_{n} \xrightarrow[n \to \infty]{\mathscr{C}^{\infty}} \mathscr{N}(0,1).
\end{equation*}
\end{proposition}
\begin{proof} 
By \cite{NPPS}, we find that the spectral radius $\rho(\mathsf{A}^{(n)}) \to 0$ as $n \to \infty$.
Let $q \in \mathbb{N}$.
By \cref{th:spectral-radius-implies-influence} this implies that $\limsup \mathcal{R}_{q}(\mathsf{A}^{(n)}) > 0$ and $\tau(\mathsf{A}^{(n)}) \to 0$.
We apply \cref{th:regularity-quadratic-form} to conclude.
\end{proof}

\subsubsection{Improving the Carbery--Wright inequality}
Assume that the $X_{i}$'s are centered, normalized, and with log-concave, and let $Q \coloneq \psh{X}{\mathsf{A} X}$ be a random quadratic form in the $X_{i}$'s.
Then, a celebrated result of \citeauthor{CarberyWright} \cite{CarberyWright} states that
\begin{equation}\label{eq:carbery-wright}
  \Prob*{ Q \leq \varepsilon } \lesssim \varepsilon^{1/2}.
\end{equation}
However, whenever $\rho(\mathsf{A})$ is small enough \cref{th:regularity-quadratic-form} ensures that the law of $Q$ has a continuous density.
Thus \cref{eq:carbery-wright} is improved to
\begin{equation}\label{eq:carbery-wright-improved}
  \Prob*{ Q \leq \varepsilon } \lesssim \varepsilon.
\end{equation}
Note that \cref{eq:carbery-wright} holds for any $(X_{i})$ whose joint law is a log-concave distribution (not necessarily independent); while \cref{eq:carbery-wright-improved} holds for sequence of independent random variables such that \cref{ass:small-ball-gamma} hold and the influence is small enough.
The use of the Carbery--Wright inequality is ubiquitous in probability theory, and an improvement of their inequality could improve several fundamental results in probability theory.
We are planning to explore in details such consequences in a future work.

\printbibliography%
\end{document}